\documentclass{l4dc2026}

\usepackage{algorithm}
\usepackage{algcompatible}
\usepackage{graphicx}
\hypersetup{
   breaklinks=true,   
   colorlinks=true,   
}

\usepackage{graphicx}
\usepackage[font=small,labelfont=bf]{caption} 
\usepackage{comment}

\newtheorem{assumption}{Assumption}

\newtheorem{thm}{Theorem}
\newtheorem{prop}{Proposition}
\newtheorem{lem}{Lemma}
\newtheorem{rem}{Remark}

\usepackage[nameinlink]{cleveref}
\crefname{equation}{}{}
\crefname{theorem}{Theorem}{Theorems}
\crefname{corollary}{Corollary}{Corollaries}
\crefname{example}{Example}{Examples}
\crefname{assumption}{Assumption}{Assumptions}
\crefname{lemma}{Lemma}{Lemmas}
\crefname{proposition}{Proposition}{Propositions}
\crefname{figure}{Figure}{Figures}
\crefname{table}{Table}{Tables}
\crefname{section}{Section}{Sections}
\crefname{appendix}{Appendix}{Appendices}
\crefname{definition}{Definition}{Definitions}
\crefname{algorithm}{Algorithm}{Algorithms}

\crefname{thm}{Theorem}{Theorems}
\crefname{prop}{Proposition}{Propositions}
\crefname{lem}{Lemma}{Lemmas}
\crefname{deff}{Definition}{Definitions}
\crefname{rem}{Remark}{Remarks}
\crefname{cor}{Corollary}{Corollaryies}

\Crefname{equation}{}{}
\Crefname{theorem}{Theorem}{Theorems}
\Crefname{corollary}{Corollary}{Corollaries}
\Crefname{example}{Example}{Examples}
\Crefname{lemma}{Lemma}{Lemmas}
\Crefname{proposition}{Proposition}{Propositions}
\Crefname{figure}{Figure}{Figures}
\Crefname{fact}{Fact}{Facts}
\Crefname{table}{Table}{Tables}
\Crefname{section}{Section}{Sections}
\Crefname{appendix}{Appendix}{Appendices}
\crefname{prop}{Proposition}{Propositions}

\newcommand{\tr}{\mathsf{T}}
\newcommand{\RR}{\mathbb{R}}

\newcommand{\Dist}{\mathrm{dist}}

\DeclareMathOperator*{\argmin}{\arg\!\min}

\newcommand{\innerproduct}[2]{\left \langle #1, #2 \right\rangle }

\newcommand{\bigO}{\mathcal{O}}

\usepackage{enumitem} 

\usepackage{wrapfig}

\allowbreak

\usepackage{graphicx}

\newcommand{\ProDes}{$\mathtt{ProxDescent}$}

\title[An accelerated PBM for convex optimization]{An accelerated proximal bundle method for convex optimization}
\usepackage{times}

\author{%
 \Name{Feng-Yi Liao} \Email{fliao@ucsd.edu }\\
 \addr Department of Electrical and Computer Engineering, University of California San Diego
 \AND
 \Name{Thomas Madden} \Email{thmadden@ucsd.edu}\\
 \addr Department of Mathematics, University of California San Diego%
 \AND
 \Name{Yang Zheng} \Email{zhengy@ucsd.edu}\\
 \addr Department of Electrical and Computer Engineering, University of California San Diego%
}

\begin{document}

\maketitle

\vspace{-3mm}
\begin{abstract}
The proximal bundle method (PBM) is a powerful and widely used approach for minimizing nonsmooth convex functions. 
However, for smooth objectives, its best-known convergence rate remains suboptimal, and whether PBM can be accelerated remains open. 
In this work, we present the first \emph{accelerated proximal bundle method} that achieves the optimal $\mathcal{O}(1/\sqrt{\epsilon})$ iteration complexity for obtaining an $\epsilon$-accurate solution in smooth convex optimization.
The proposed method is \textit{conceptually simple}, and differs from Nesterov’s accelerated gradient descent by only a single line and retains all key structural properties of the classical PBM.
In particular, it relies on the same minimal assumptions on model approximations and preserves the standard bundle testing criterion. 
Numerical experiments confirm the accelerated $\mathcal{O}(1/\sqrt{\epsilon})$ convergence rate predicted by our theory.
\end{abstract}


\section{Introduction}
Convex optimization plays a fundamental role in various disciplines \citep{nesterov2018lectures,boyd2004convex}. 
Among its many subfields, \textit{smooth and convex} optimization is arguably the most fundamental, which considers the unconstrained problem 
\begin{equation}
     \label{eq:opt-pb}
    f^\star = \min_{x \in \RR^n} \; f(x)
\end{equation}
where $f:\RR^n \to \RR$ is  convex and $M$-smooth, i.e., $\| \nabla f(x) - \nabla f(y) \| \leq M \|x - y\|, \forall \ x,y  \in \RR^n $. In this setting, the simplest yet fundamental algorithm is the gradient descent, which moves along the negative gradient direction at every iteration, i.e., $x_{k+1} = x_k - \alpha_k \nabla f(x_k),$ where $\alpha_k$ denotes the stepsize. A standard textbook result shows that with a constant step size $\alpha_k = \frac{1}{M}$, gradient descent can find an $\epsilon$-solution (i.e., an iterate $x_k$ such that $f(x_k) - f^\star \leq \epsilon$) within  $\bigO(1/\epsilon)$ iterations (\citet[Section~1.2.3]{nesterov2018lectures}). Nevertheless, the rate is only suboptimal. Nesterov introduced an accelerated gradient descent (AGD), which improves the iteration complexity to the optimal complexity $\bigO(1/\sqrt{\epsilon})$ (\citet[Section~2.2]{nesterov2018lectures}). 

Another conceptually simple algorithm is the \textit{proximal point method} (PPM), which follows an implicit (sub)gradient update. Equivalently, PPM can be viewed as applying gradient descent to the Moreau envelope (\citet[Section 4.1.1]{parikh2014proximal}). It achieves the same $\bigO(1/\epsilon)$ iteration complexity for both smooth and nonsmooth convex functions (\citet[Theorem 2.1]{guler1991convergence}). Inspired by Nesterov's AGD, various accelerated PPMs have also been proposed to improve the complexity to $\bigO(1/\sqrt{\epsilon})$; see e.g., \citep{he2012accelerated, monteiro2013accelerated, salzo2012inexact}.  

Although PPM can handle both smooth and nonsmooth objectives, its update is often computationally intractable, as it requires solving a proximal subproblem exactly. 
To address this limitation, the \textit{proximal bundle method} (PBM), originally introduced in \citep{lemarechal1978nonsmooth,mifflin1977algorithm} for nonsmooth functions, relaxes the exact proximal step to an inexact implicit (sub)gradient update. This relaxation greatly reduces computational cost while preserving the convergence properties of PPM. The asymptotic convergence of PBM iterates to an optimal solution was first established in \citet{kiwiel1983aggregate}.
Subsequently, \citet{kiwiel2000efficiency} provided the first nonasymptotic complexity bound of $\mathcal{O}(1/\epsilon^3)$ for obtaining an $\epsilon$-solution of general convex (potentially nonsmooth) functions. More recently, sharper convergence guarantees have been established under additional growth or smoothness assumptions \citep{du2017rate, diaz2023optimal}.
Nevertheless, even for smooth convex functions, the best-known iteration complexity of PBM remains suboptimal at $\mathcal{O}(1/\epsilon)$, even with adaptive step-size rules \citep[Table~1]{diaz2023optimal}.

It remains unclear whether and how the classical PBM can be accelerated to achieve the optimal rate $\mathcal{O}(1/\sqrt{\epsilon})$.  
As noted in two recent studies  \citep{diaz2023optimal, liang2024unified}, this question is still open. One difficulty lies in the convergence analysis itself: even for the standard PBM, establishing tight iteration-complexity bounds has been challenging and was only recently clarified in \cite{diaz2023optimal}. Another challenge is algorithmic: unlike gradient-based methods, it is not obvious how to introduce momentum or extrapolation into the traditional single-loop PBM framework. Very recently, \cite{fersztand2025acceleration} achieved an improved complexity of $\mathcal{O}(\log(1/\epsilon)/\sqrt{\epsilon})$ by employing more intricate bundle model assumptions.

In this work, we introduce the \textit{first} accelerated PBM that achieves the optimal iteration complexity $\bigO(1/\sqrt{\epsilon})$ for finding an $\epsilon$-solution for smooth convex functions.~Algorithmically, our development of accelerated PBM is derived intuitively from Nesterov's famous AGD \citep{nesterov1983method}. In particular, our accelerated PBM only differs by one line from Nesterov's AGD. Under a proper choice of parameters and under-estimators, our proposed method reduces to Nesterov's AGD exactly. Moreover, our proposed accelerated PBM preserves \textit{all} key features of the classical PBM, as it \textit{does not} change the testing criterion of the classical PBM or impose any additional assumptions on the under-estimators. This is made possible through a crucial observation that the proximal bundle update can be interpreted as an inexact implicit (sub)gradient step, which naturally accommodates Nesterov-type momentum and enables acceleration. 
Importantly, our proposed accelerated PBM can be viewed as a special realization of the abstract accelerated inexact proximal point framework in \cite{monteiro2013accelerated}. We present a detailed comparison with \cite{fersztand2025acceleration} in \Cref{remark:comparison}.

The rest of this paper is organized as follows. \Cref{sec:prelim} reviews the proximal point method and the PBM. \cref{sec:ACC-PBM} introduces our proposed accelerated PBM. \cref{sec:proofs} establishes the convergence guarantees. \cref{sec:Numerical} shows numerical experiments, and \cref{sec:conclusion} concludes the paper.

\section{Preliminaries and problem statement} 
\label{sec:prelim}

\subsection{Gradient descent and proximal point methods} 
Consider the optimization problem \cref{eq:opt-pb}. If the objective $f$ is a differentiable convex function with $M$-Lipschitz continuous gradient, the standard \emph{gradient descent} (GD) method performs the update
\begin{equation} \label{eq:gradient-update}
x_{k+1} = x_k - \alpha_k\nabla f(x_k),
\end{equation}
where $\alpha_k$ denotes the step size. With the constant stepsize $\alpha_k = {1}/{M}$, the GD iterates converge with a rate $f(x_k) -f^\star = \mathcal{O}(1/k)$; see, e.g., Corollary 2.1.2 in \cite{nesterov2018lectures}

For a general convex (possibly nonsmooth) objective $f$, the \emph{proximal point method} (PPM) replaces the gradient step \cref{eq:gradient-update} with a proximal update 
\begin{equation} \label{eq:PPM-update}
y_{k+1} = \arg\min_{x}\, \left\{f(x) + \frac{\rho}{2}\|x - y_k\|^2\right\}, \qquad \rho>0. 
\end{equation}
The PPM update \cref{eq:PPM-update} is well-defined for any $\rho >0$ since the subproblem is strongly convex and thus admits a unique solution. It is known that the PPM has the same convergence rate $f(y_k) - f^\star = \bigO(1/k)$; see e.g., \citet[Theorem 2.1]{guler1991convergence}. %
This rate holds for both smooth and~nonsmooth convex objectives. In general, the PPM update \cref{eq:PPM-update} may not be realizable efficiently. Nevertheless, the PPM serves as a conceptual foundation for modern proximal and bundle-type algorithms \citep{drusvyatskiy2017proximal,liang2021proximal,diaz2023optimal,liao2025bundle,liao2025proximal}.

\subsection{Proximal bundle method as an inexact PPM}
\label{subsec:PBM}

\begin{wrapfigure}[7]{r}{0.44\textwidth} 
  \vspace{-8.5mm}                
  
  \begin{minipage}{\linewidth}
   \begin{algorithm}[H]
\caption{\small Proximal Bundle method (PBM)}\label{alg:bundle}
\begin{algorithmic}[1]
\REQUIRE $y_0 \in \RR^n, \rho > 0, \beta \in (0,1)$, $k_{\max}\in \mathbb{N}$
\FOR{$k=0,1,2, \ldots,k_{\max}$}
    \STATE $y_{k+1} = $ \ProDes($y_k,\beta,\rho$); 
\ENDFOR
\end{algorithmic}
\end{algorithm}
  \end{minipage}
\end{wrapfigure}

It is known that the PPM \cref{eq:PPM-update} can be interpreted as an algorithm which performs an \textit{implicit} (sub)gradient update \citep{correa1993convergence}. 
From the optimality condition, \cref{eq:PPM-update} can be equivalently written as   
\begin{align} \label{eq:PPM-gradient}
    y_{k+1} = y_k - \frac{1}{\rho} g_{k + 1},
\end{align}
where $g_{k + 1} \in \partial f(y_{k+1})$ is a subgradient at the next iterate $y_{k+1}$. Recall that the usual convex subdifferential is defined as $\partial f(x) :=  \{s \in \RR^n \mid f(y) \geq f(x) + \innerproduct{s}{y-x}, \forall y \in \RR^n \}$. If the function $f$ is differentiable, the subdifferential reduces to the usual gradient, i.e., $ \partial f(x) = \{\nabla f(x)\}$. 

Since the PPM update \cref{eq:PPM-update} or \cref{eq:PPM-gradient} is non-trivial, inexact variants of the PPM have been widely considered \citep{rockafellar1976monotone}. We here introduce the class of proximal bundle methods (PBMs)  \citep{kiwiel2000efficiency} as a special inexact PPM. We list the PBM as a double-loop algorithm in \Cref{alg:bundle} involving a subroutine \ProDes($y_k,\beta,\rho$), which we shall define next.  

We recall the notion of $\epsilon$-inexact subdifferential for a convex function $f$:
\begin{equation} \label{eq:inexact-subdifferential}
\partial_{\epsilon} f(x) := \{ v \in \RR^n \mid f(y) \geq f(x) + \innerproduct{v}{y - x} - \epsilon, \forall \ y \in \RR^n\},
\end{equation}
where $\epsilon \geq 0$. It is clear that we have $\partial_{\epsilon} f(x) = \partial f(x)$ when $\epsilon = 0$. Instead of the true subgradient update \cref{eq:PPM-gradient}, the PBM considers an inexact update 
\begin{align} \label{eq:PPM-gradient-inexact}
    y_{k+1} = y_k - \frac{1}{\rho} \tilde{g}_{k + 1},
\end{align}
where $\tilde{g}_{k + 1}\! \! \in\! \!\partial_{\epsilon_{k+1}} f(y_{k+1})$ is an inexact subgradient at the next iterate $y_{k+1}$ with inexactness $\epsilon_{k+1} \! \geq \! 0$ coming from a descent criterion that depends on a pre-defined parameter $\beta \! \in \! (0,1]$; see \cref{eq:testing-bundle}. 

In general, finding such an inexact subgradient $\tilde{g}_{k + 1}$ requires an iterative search. For brevity, we denote such a subroutine as \texttt{$y_{k+1}=$ \ProDes$(y_k,\beta,\rho)$}, since the iterate $y_{k+1}$ comes from simple proximal updates (see \cref{eq:PBM-trial-point}) and satisfies a certain descent performance (see \cref{eq:testing-bundle}).

In particular, \ProDes($y_k,\beta,\rho$) is realized efficiently via an iterative bundle strategy, as detailed below.  

\begin{wrapfigure}[10]{r}{0.46\textwidth} 
  \vspace{-5.5mm}                
  
  \begin{minipage}{\linewidth}
   \begin{algorithm}[H]
\caption{\ProDes($y_k,\beta,\rho$)}\label{alg:Proxi-descent-subproblem}
\begin{algorithmic}[1]
\STATE Initialize $z_1  = y_k$ 
\FOR{$j=1,2,\ldots$}
    \STATE Construct $\tilde{f}_j$ satisfying \cref{assump:bm};
    \STATE Compute $z_{j+1}$ using \cref{eq:PBM-trial-point};  
    \IF{\cref{eq:testing-bundle} holds} Break;
    \ENDIF
\ENDFOR
 \STATE \textbf{Return} $z_{j+1}$;  
\end{algorithmic}
\end{algorithm}
  \end{minipage}
\end{wrapfigure}

\vspace{2mm}

\noindent \textbf{Proximal bundle strategy:} A unique feature of the PBM is to approximate the original function $f$ such that the proximal update becomes much simpler. Starting with $j = 1$ and $z_1 = y_k$, the PBM generates a sequence of candidate points 
\begin{align}
    \label{eq:PBM-trial-point}
    z_{j+1} = \argmin_{y}\; \left\{\tilde{f}_{j}(y) + \frac{\rho}{2}\|y - y_k\|^2 \right\}, 
\end{align}
where $\tilde{f}_j$ is a convex under-estimator of $f$. Since $f$ is convex, such an under-estimator  $\tilde{f}_j$ is easy to construct, e.g., using its (sub)gradient $s\in \partial f(y_k)$. 
If this candidate point $z_{j+1}$ satisfies
\begin{equation}
    \label{eq:testing-bundle}
    f(y_{k}) - f(z_{j+1}) \geq \beta(f(y_{k}) -  \tilde{f}_j (z_{j+1})), 
\end{equation}
where $\beta \in (0,1]$, then we set $y_{k+1} = z_{j+1}$; otherwise, we construct a new convex under-estimator $\tilde{f}_{j+1}$ satisfying \cref{assump:bm}, and repeat \cref{eq:PBM-trial-point} with $\tilde{f}_{j+1}$ until \cref{eq:testing-bundle} is satisfied. 
We list this procedure in \Cref{alg:Proxi-descent-subproblem}. The criterion \cref{eq:testing-bundle} indicates that the candidate $z_{j+1}$ achieves at least $\beta$ fraction of the objective value decrease that the model $f_j(\cdot)$ predicts. If  \cref{eq:testing-bundle} holds, this is known as a \textit{descent} step (and \Cref{alg:Proxi-descent-subproblem} terminates); otherwise, it is a \textit{null} step. 
At this stage, it is not obvious how such a criterion \cref{eq:testing-bundle} is connected with the inexact update \cref{eq:PPM-gradient-inexact}. We will detail their relationship in \Cref{sec:proofs}, which turns out to be one key observation that leads to our accelerated PBM algorithm. 

\vspace{-2mm}

\begin{assumption}
    \label{assump:bm} 
    Let $z_{1} = y_k$. For $j \geq 1$, the convex function $f_{j}$ satisfies three conditions: 
    \vspace{-2mm}
    \begin{enumerate}[leftmargin=*]
    \setlength{\itemsep}{0pt}
        \item \textbf{Lower approximation:} $\tilde{f}_{j}(y) \leq f(y), \forall y \in \mathbb{R}^n.$ \vspace{-1pt}
        \item \textbf{Subgradient:} There exists $ g_{j} \in \partial f(z_{j})$ such that 
        $\tilde{f}_{j}(y) \geq f(z_{j}) + \innerproduct{g_{j}}{y-z_{j}}, \forall y \in \RR^n.$ 
        \item \textbf{Aggregation:} If $j > 1$ (equivalently, \cref{eq:testing-bundle} does not hold for $z_{j}$), we require
        $\tilde{f}_{j}(y) \geq \tilde{f}_{j-1}(z_{j}) + \innerproduct{s_{j}}{y-z_{j}},\forall y \in \RR^n,$
        where $s_{j} = \rho(y_k - z_{j}) \in \partial  \tilde{f}_{j-1}(z_{j})$.
    \end{enumerate}
\end{assumption}

\vspace{-2mm}

In \cref{assump:bm}, the first requirement means that $\tilde{f}_{j}$ is a global under-estimator of $f$, the second requirement means that $\tilde{f}_{j}$ should be a better model of $f$ than the subgradient lower-bound of $f$ at $z_{j}$, and the last requirement asks that $\tilde{f}_{j}$ is also better than the linearization of the previous model $\tilde{f}_{j-1}$ at $z_j$. The second and the third requirements together guarantee that the value $\eta_j:=\min_{y} \{\tilde{f}_{j}(y) + \frac{\rho}{2}\|y - y_k\|^2 \}$ increases monotonically as $j$ increases, and the first requirement guarantees that this value $\eta_j$ is upper-bounded as  $\eta_j \leq \eta^*:= \min_{y} \{f(y) + \frac{\rho}{2}\|y - y_k\|^2 \}$ for all $j\geq 1$.  

Any approximation model $\tilde{f}_j$ satisfying \cref{assump:bm} can ensure that \cref{alg:Proxi-descent-subproblem} terminates with a finite number of iterations for general convex functions; see e.g., \cite{diaz2023optimal}. Specifically, we can choose $\tilde{f}_j$ as the maximum of two linear lower bounds:  
\begin{equation}
\label{eq:finite-memory-models}
    \tilde{f}_{j}(y) = \max\{ f(z_{j}) + \innerproduct{g_{j}}{y-z_{j}}, \tilde{f}_{j-1}(z_{j}) + \innerproduct{s_{j}}{y-z_{j}}\}.
\end{equation}
In this case, the proximal update \cref{eq:PBM-trial-point} admits an analytically closed-form solution, which is almost as efficient as the simple (sub)gradient update \cref{eq:gradient-update}. 
For specific applications, we may choose other approximating models \citep{helmberg2000spectral,liao2025overview}.  
If $f$ is a smooth convex function, \cref{alg:Proxi-descent-subproblem} will terminate within a constant number of iterations.

\vspace{-2mm}

\begin{lem}[{\citet[Lemma 5.2 and Theorem 2.2]{diaz2023optimal}}] \label{lemma:descent-step}
    Suppose $f$ is convex and $M$-smooth. Let $\beta \in (0,1)$ and $\rho>0$. Then, for any $y_k \in \mathbb{R}^n$, \texttt{ProxDescent($y_k,\beta,\rho$)} in \cref{alg:Proxi-descent-subproblem} terminates in at most $\frac{16(M+\rho)^3}{(1-\beta)^2 \rho^3}$ iterations. Accordingly, the total iteration complexity of \Cref{alg:bundle} (including all inner iterations) to find an $\epsilon$-optimal solution is $\mathcal{O}(1/\epsilon)$.
\end{lem}

\vspace{-2mm}

    The PBM described in \cref{alg:bundle} is presented as a double-loop algorithm. Historically, however, the PBM is often formulated in a single-loop form; see, e.g., \citet[Algorithm 2.1]{kiwiel2000efficiency} and \citet[Algorithm 1]{diaz2023optimal}. These two formulations are in fact equivalent; see \citet[Section 4.3]{liao2025proximal} or \cref{subsec:reinterpretation}. The double-loop form in \cref{alg:bundle} will prove to be particularly convenient for developing the accelerated PBM in this work.

\subsection{Problem statement}
The $\bigO(1/\epsilon)$ complexity of PBM established in \cref{lemma:descent-step} matches that of the exact PPM (\citet[Theorem 2.1]{guler1991convergence}) and vanilla gradient descent with a constant stepsize (\citet[Corollary 2.1.2]{nesterov2018lectures}). However, this $\bigO(1/\epsilon)$ rate is not optimal for smooth convex functions. For this class, it is well-known that gradient descent with momentum (commonly referred to as \textit{accelerated gradient descent} \citep{nesterov1983method}) achieves the optimal rate $\bigO(1/\sqrt{\epsilon})$. Similarly, the PPM can be accelerated to achieve the same optimal rate, with the first such result due to \citep{doi:10.1137/0802032}. Building on this foundation, several accelerated \emph{inexact} PPMs have been developed in \citep{he2012accelerated, monteiro2013accelerated, salzo2012inexact, lin2018catalyst}.

In this work, we aim to develop an accelerated PBM that attains the optimal convergence rate $\mathcal{O}(1/\sqrt{\epsilon})$ for smooth convex problems.
The proposed method will rely only on the standard bundle assumptions in \Cref{assump:bm} and the conventional descent criterion \cref{eq:testing-bundle}, without requiring any additional bundle modifications.
Our key observation is that the inexact (sub)gradient update \cref{eq:PPM-gradient-inexact} and the double-loop structure in \cref{alg:bundle} can naturally incorporate Nesterov’s momentum mechanism.
Importantly, the resulting accelerated PBM can be interpreted as a particular instance of the accelerated inexact proximal point framework introduced by \cite{monteiro2013accelerated}.

\section{An accelerated proximal bundle method}
\label{sec:ACC-PBM}
This section introduces our accelerated PBM that integrates Nesterov’s extrapolation mechanism. We present its convergence guarantees and identify an interesting connection with Nesterov’s~AGD.

\subsection{Integrating Nesterov's extrapolation with PBM}
\label{subsec:integration}

Among various acceleration techniques for first-order methods, the most fundamental and influential is \emph{Nesterov's accelerated gradient descent} (AGD) \citep{nesterov1983method, nesterov2018lectures}. Note that Nesterov's AGD has multiple equivalent formulations (see e.g., \citet[Section 4.4]{d2021acceleration}). We list a particular Nesterov's AGD in \Cref{alg:NGD}.

Unlike the standard gradient descent \cref{eq:gradient-update}, Nesterov's AGD maintains three coupled sequences of iterates, denoted by $\{x_k\}$, $\{y_k\}$, and $\{z_k\}$. 
At each iteration, it performs a gradient step at the extrapolated point $y_k$ (see line \ref{alg:AGD-gradient-step} in \Cref{alg:NGD}), followed by a momentum-based extrapolation to generate the next iterate $y_{k+1}$ in lines \ref{alg:AGD-extrapolation-1} and \ref{alg:AGD-extrapolation-2}.   
This additional extrapolation step introduces a form of momentum that substantially accelerates the convergence performance. Indeed,  \Cref{alg:NGD} improves the iteration complexity from $\mathcal{O}(1/\epsilon)$ to $\mathcal{O}(1/\sqrt{\epsilon})$ in achieving an $\epsilon$-accurate solution for smooth convex functions \cite[Theorem 2.17]{nesterov2018lectures}.

Our key observation is that the PBM can be interpreted as an inexact (sub)gradient update \cref{eq:PPM-gradient-inexact}. It is then possible to integrate this Nesterov’s extrapolation mechanism within the standard PBM framework. 
In particular, leveraging the double-loop structure in \cref{alg:bundle}, we propose the accelerated PBM listed in \cref{alg:A-PBM-main}. Notably, \cref{alg:NGD,alg:A-PBM-main} share the same extrapolation and momentum updates; they differ only in the sixth line, where the exact gradient step in AGD 
$
x_{k+1} = y_k - \nabla f(y_k)/M
$ is replaced by the PBM oracle 
$$x_{k+1} = \text{\ProDes}(y_k,\beta,\rho).$$
This oracle corresponds to the inexact proximal step described in \cref{eq:PPM-gradient-inexact} and ensures the descent and model-approximation conditions defined in \cref{eq:testing-bundle} and \cref{assump:bm}.

The proposed accelerated PBM can be viewed as an accelerated inexact PPM realized through standard bundle updates, which preserves the classical PBM structure. We show next that \Cref{alg:A-PBM-main} not only achieves the optimal $\mathcal{O}(1/\sqrt{\epsilon})$ convergence rate for smooth convex problems, but also naturally includes \Cref{alg:NGD} as a special case when choosing the proximal parameter $\rho$ and descent parameter $\beta$ appropriately.

\hspace{3mm}
\begin{minipage}{0.4\textwidth}
    \begin{algorithm}[H]
\caption{Nesterov's AGD}
\label{alg:NGD}
\begin{algorithmic}[1]
\STATE Set $z_0 = x_0$ and $A_0 = 0$.
\FOR{$k=0,1, \ldots, T$} 
     \STATE  $a_{k} = {(1 + \sqrt{1 + 4 A_k})}/{2}$
    \STATE $A_{k+1} = A_k + a_{k }$ 
    \STATE $y_k = \frac{A_k}{A_{k+1}} x_k + \frac{a_{k}}{A_{k+1}} z_k$ \label{alg:AGD-extrapolation-1}
    \STATE $x_{k+1} =  y_k \!-\! \nabla f(y_k)/M$ \label{alg:AGD-gradient-step}
    \State $z_{k+1} = z_k\! -\! a_k (y_k\! -\! x_{k+1})$ \label{alg:AGD-extrapolation-2}
\ENDFOR
\end{algorithmic}
\end{algorithm}
\end{minipage}
\hfill 
\begin{minipage}{0.43\textwidth}

\begin{algorithm}[H]
\caption{Accelerated PBM}
\label{alg:A-PBM-main}
\begin{algorithmic}[1]
\STATE Set $z_0 = x_0, A_0 = 0$ 
\FOR{$k=0,1, \ldots, T$} 
 \STATE $a_{k } = {(1 + \sqrt{1 + 4 A_k })}/{2}$ 
    \STATE $A_{k+1} = A_k + a_{k }$  
    \STATE $y_k = \frac{A_k}{A_{k+1}} x_k + \frac{a_{k}}{A_{k+1}} z_k$ 
    \STATE $x_{k+1} =  $ 
    \ProDes($y_k,\beta,\rho$)
    \STATE $z_{k+1} = z_k \!-\! a_k (y_k\! -\! x_{k+1})$ 
\ENDFOR
\end{algorithmic}
\end{algorithm}
\end{minipage}
\hspace{5mm}

\vspace{4mm}

\subsection{Accelerated rate for smooth convex problems}
\label{subsec:ACC-PBM}

We now present our main convergence guarantee. Similar to Nesterov's AGD, our accelerated PBM also achieves the $\bigO(1/\sqrt{\epsilon})$ optimal iteration complexity for smooth convex problems.
\begin{thm}
    \label{thm:A-PBM}
    Consider the problem \cref{eq:opt-pb}.  Let $S =  \argmin_{x} f(x)$, and assume $S \neq \emptyset$. If $f$ is convex and $M$-smooth, then \Cref{alg:A-PBM-main} with parameters $\rho \geq \frac{M}{c}$ for some $c > 0$,  and $\beta \in [ \frac{c+2\sqrt{c} +2}{c + 2\sqrt{c} + 3} ,1)$ generates a sequence of iterates $\{x_k\}$ satisfying
    \begin{align} \label{eq:PBM-acceleration}
        f(x_k) - f^\star \leq \frac{2\rho\Dist^2(x_0,S)}{k^2}, \; \forall k \geq 1.
    \end{align}
    Moreover, the total iteration count (including the inner loop) required to obtain an $\epsilon$-optimal solution, i.e., a point $x_k$ such that $f(x_k) - f^\star \le \epsilon$, is bounded by 
    \begin{equation}  \label{eq:PBM-acceleration-complexity}
        \frac{16(M+\rho)^3}{(1-\beta)^2 \rho^3} \sqrt{\frac{2\rho \cdot \Dist^2(x_0,S)}{\epsilon}}. 
    \end{equation}
\end{thm}

The proof builds on the interpretation of the PBM update as an inexact (sub)gradient step \cref{eq:PPM-gradient-inexact}, realized by the subroutine \ProDes($y_k,\beta,\rho$).  In particular, we show that the bundle-model assumptions in \cref{assump:bm} and the descent test \cref{eq:testing-bundle} jointly control the inexactness parameter $\epsilon_{k+1}$ in \cref{eq:PPM-gradient-inexact}. This connection allows us to interpret our accelerated PBM in \Cref{alg:A-PBM-main} as a special realization of the general accelerated inexact proximal point framework in \cite{monteiro2013accelerated}. 
Consequently, the accelerated rate in \cref{eq:PBM-acceleration} follows directly from \citet[Theorem~3.8]{monteiro2013accelerated}.  
The total iteration bound in \cref{eq:PBM-acceleration-complexity} then follows from \Cref{lemma:descent-step}, which ensures that each call to \ProDes($y_k,\beta,\rho$) terminates after a constant number of inner iterations. The proof details are given in \cref{sec:proofs}.

Although the convergence proof builds upon the general inexact acceleration framework in \cite{monteiro2013accelerated}, we believe the underlying observation is nontrivial.  
The possibility of accelerating PBM had remained open in the literature \citep{diaz2023optimal, liang2024unified}. 
To our best knowledge, \cref{alg:A-PBM-main} provides the first accelerated PBM that achieves the optimal convergence rate for smooth convex functions. In the past, PBMs have been primarily studied for nonsmooth optimization problems, and this reflects their origins in cutting-plane and bundle methods. Only recently, an $\bigO(1/\epsilon)$ iteration complexity of PBM for smooth functions was established in \cite{diaz2023optimal}. 
The result in \Cref{thm:A-PBM} advances this line of work by improving the rate from $\mathcal{O}(1/\epsilon)$ to $\mathcal{O}(1/\sqrt{\epsilon})$ through the incorporation of Nesterov-type momentum and extrapolation, paralleling the developments of Nesterov’s AGD and the accelerated PPM.

We note that Nesterov's AGD in \cref{alg:NGD} always requires the small stepsize choice $1/M$. In contrast, as guaranteed in \cref{thm:A-PBM}, the accelerated PBM in \cref{alg:A-PBM-main} allows a larger stepsize $1/\rho$ since $\rho > 0$ can be chosen arbitrarily, as long as we choose an appropriate $\beta$ to control the solution quality of a descent step.

\vspace{-2mm}

\begin{rem} 
\label{remark:comparison}
    Here, we compare our accelerated PBM (\Cref{alg:A-PBM-main}) with the recent development \citet[Algorithm 3]{fersztand2025acceleration}. First, our \Cref{alg:A-PBM-main} does not change the descent criterion for the PBM, whereas \cite{fersztand2025acceleration} proposes to use a more stringent testing criterion. Second, our \cref{assump:bm} for the under-estimators $\tilde{f}_j$ are standard in the classical PBM framework. In contrast, \cite{fersztand2025acceleration} requires a \textit{smooth} under-estimator, which is not satisfied by the usual cutting-plane approximation (e.g. $\tilde{f}_{j}(y) = \max_{i = 1,\ldots,j} \{f(y_i) + \innerproduct{\nabla f(y_i)}{y - y_i}\}$). Third, our result is more general as the under-estimator assumption in \cite{fersztand2025acceleration} naturally satisfies our \cref{assump:bm}. As a result, our algorithm enables the use of finite-memory models, where the difficulty of performing a null step is limited by a (sub)gradient evaluation. Finally, our $\bigO(1/\sqrt{\epsilon})$ iteration complexity in \cref{thm:A-PBM} is better than $\bigO(\log(1/\epsilon)/\sqrt{\epsilon})$ in \citet[Theorem 26]{fersztand2025acceleration} by a logarithmic term.
\end{rem}

\subsection{Connections with Nesterov's AGD}

Finally, we point out an interesting connection with Nesterov's AGD.
For smooth convex functions, the oracle \ProDes($y_k,\beta,\rho$) can be tuned to produce an iterate in a single inner iteration when the parameters $\beta$ and $\rho$ are chosen appropriately. Similar observations have been made in \citep{ding2023revisiting,liao2025overview} in the context of applying PBM to solve semidefinite programs.

\begin{prop}
    \label{prop:connection-NGD}
    Let $f:\RR^n \to \RR$ be a $M$-smooth function. Given a center point $y_k \in \RR^n$, suppose that we choose $\beta \in (0, \frac{1}{2}]$ and $\rho \geq M$ in \Cref{alg:Proxi-descent-subproblem}. Then, \Cref{alg:Proxi-descent-subproblem} terminates in one iteration.  
\end{prop}

Due to the page limit, the proof is provided in \cref{subsec:connection-NGD-pf}.  \cref{prop:connection-NGD} implies that if the first under-estimator is constructed as $\tilde{f}_1(\cdot) = f(y_k) + \innerproduct{\nabla f(y_k)}{\cdot-y_k}$ and parameters are chosen as $\beta\in (0,1/2]$ and $\rho \geq M$, then \cref{alg:Proxi-descent-subproblem} will directly output the usual gradient update 
$$z_2 = y_k - \textstyle \frac{1}{\rho}\nabla f(y_k)$$ 
since $z_{2} = \argmin_{y}\{ \tilde{f}_{1}(y) + \frac{\rho}{2}\|y - y_k\|^2 \}$. In this case, our \cref{alg:A-PBM-main} reduces exactly to Nesterov's AGD in \cref{alg:NGD}. When choosing $\beta \in (1/2,1)$, the oracle \ProDes($y_k,\beta,\rho$) may take more (but a constant number of) gradient-like iterations before the extrapolation step.

In this sense, our proposed \cref{alg:A-PBM-main} gracefully generalizes Nesterov's AGD by allowing multiple gradient-like updates within each proximal step before applying the extrapolation mechanism.
Such a tunable integration of multiple gradient-like iterations within a Nesterov-type scheme seems not to have been explicitly established before. 

\vspace{-2mm}

\section{Technical proofs}
\label{sec:proofs}

In this section, we prove the accelerated convergence of our proposed \cref{alg:A-PBM-main}, stated in \cref{thm:A-PBM}. In particular, we show that \cref{alg:A-PBM-main} can be viewed as a special realization of the accelerated hybrid proximal extragradient (A-HPE) in \cite{monteiro2013accelerated}.

\vspace{-1mm}
\subsection{A simplified A-HPE algorithm}

\begin{wrapfigure}[10]{r}{0.48\textwidth}
\vspace{-11mm}                
\begin{minipage}{0.48\textwidth}
\begin{algorithm}[H]
\caption{Simplified~A-HPE}
\label{alg:A-HPE-simplified}
\begin{algorithmic}[1]
\STATE Set $z_0 = x_0, A_0 = 0, \rho > 0$.
\FOR{$k=0,1, \ldots, T$} 
 \STATE $a_{k } = \frac{1 + \sqrt{1 + 4 A_k }}{2}$ 
    \STATE $A_{k+1} = A_k + a_{k }$  
    \STATE $y_k = \frac{A_k}{A_{k+1}} x_k + \frac{a_{k}}{A_{k+1}} z_k$ 
    \STATE $x_{k+1} \!=\! y_k \!- g_{k + 1}/\rho$ satisfying \cref{eq:A-HPE-condition} \label{alg:acceleration-gradient}
    \STATE $z_{k+1} = z_k \!- \! a_k (y_k\! -\! x_{k+1})$ 
\ENDFOR
\end{algorithmic}
\end{algorithm}
\end{minipage}
\end{wrapfigure}

Let us first review a simplified and reformulated version of A-HPE, listed in \cref{alg:A-HPE-simplified}. We note that the update of \cref{alg:A-HPE-simplified} is slightly different from the original version in \cite{monteiro2013accelerated}. We believe this reformulated version may provide a clearer picture of the connection with Nesterov's AGD.

The only difference between \cref{alg:A-HPE-simplified,alg:NGD} is the update in the sixth line. In \cref{alg:A-HPE-simplified}, instead of restricting to a simple gradient update, \cref{alg:A-HPE-simplified} uses an update direction $g_{k + 1} \in \RR^n$ with an inexactness $\epsilon_{k+1} \geq 0$ such that 
\begin{align}
    \label{eq:A-HPE-condition}
    x_{k+1} = y_k - \frac{1}{\rho} g_{k + 1}, \;\; g_{k + 1} \in \partial_{\epsilon_{k+1}}f(x_{k+1}),  \;\; 2 \epsilon_{k+1} \leq \rho \| x_{k+1} - y_k\|^2.
\end{align}
This scheme can be viewed as an inexact PPM update, as the inexactness $\epsilon_{k+1}$ is expressed in terms of the subdifferential at the next iterate. Below, we state the convergence of \cref{alg:A-HPE-simplified}.

\vspace{-2mm}
\begin{thm}{\citet[Theorem 3.8]{monteiro2013accelerated}} \label{thm:monteiro-conv-main} Let $f:\RR^n \to \RR$ be a convex function and set $S =  \argmin_{x} f(x)$. Assume $S \neq \emptyset$. The sequence $\{x_k\}$ in \cref{alg:A-HPE-simplified} satisfies
    \begin{equation*}
        f(x_k) - \min_{x} f(x) \leq \frac{2 \rho \cdot \Dist(x_0, S)^2}{k^2}, \; \forall k\geq 1.
    \end{equation*}
\end{thm}
For self-completeness, we provide the proof of \cref{thm:monteiro-conv-main} in \cref{sec:A-HPE-simplified-conv}. An alternative proof using a potential function argument is provided in \cref{section:potential}. Note that \cref{thm:monteiro-conv-main} holds for both smooth and nonsmooth convex functions. However, for a nonsmooth function, condition \cref{eq:A-HPE-condition} may be difficult to satisfy unless a strong oracle is assumed, which inevitably requires more computation.

\subsection{Proof of \Cref{thm:A-PBM}}

We now establish the convergence of \cref{alg:A-PBM-main}. We only need to show that the iterate $x_{k+1} =$  \ProDes($y_k,\beta,\rho$) satisfies \cref{eq:A-HPE-condition}. Recall that \ProDes($y_k,\beta,\rho$) is realized by repeating 
\begin{align}
    \label{eq:subproblem}
    z_{j+1} = \argmin_{y}\; \left\{\tilde{f}_{j}(y) + \frac{\rho}{2}\|y - y_k\|^2 \right\}
\end{align}
until a point $z_{j+1}$ satisfies \cref{eq:testing-bundle}. Our first observation is that $z_{j+1}$ can be rewritten in the form of \cref{eq:A-HPE-condition}.
\begin{prop}
    \label{prop:inexact-subgrad}
    Let $f : \RR^n \to \RR$ be convex. 
    For every $j\geq 1$ iteration of~\ProDes($y_k,\beta,\rho$) (i.e. \cref{alg:Proxi-descent-subproblem}) with $\rho > 0$ and $\beta \in (0,1)$, it holds that $ z_{j+1} = y_k - \frac{1}{\rho} \tilde{g}_{j + 1}, \; \tilde{g}_{j + 1} \in \partial f_{\tilde{\epsilon}_{j+1}}(z_{j+1})$ with the inexactness satisfying $   \tilde{\epsilon}_{j+1} =f(z_{j+1}) - \tilde{f}_{j}(z_{j+1}).$ 
\end{prop}
\begin{proof}
    From the optimality condition, \cref{eq:subproblem} can be written as  
\begin{align*}
    z_{j+1} = y_k -  \tilde{g}_{j + 1}/\rho, \; \tilde{g}_{j + 1}  \in \partial \tilde{f}_{j} (z_{j+1}).
\end{align*}
The fact that $\tilde{g}_{j + 1} \in \partial \tilde{f}_{j} (z_{j+1})$ implies that for all $y$, we have the following:
\begin{align*}
    f(y)  & \geq \tilde{f}_{j}(z_{j+1}) +  \innerproduct{\tilde{g}_{j + 1}}{ y - z_{j+1} }  \\
    & = f(z_{j+1}) + \innerproduct{\tilde{g}_{j + 1}}{ y - z_{j+1} } - (f(z_{j+1})  - \tilde{f}_{j}(z_{j+1})).
\end{align*}
This is the same as saying that $\tilde{g}_{j + 1} \in \partial_{\tilde{\epsilon}_{j+1}} f(z_{j+1})$ with $\tilde{\epsilon}_{j+1} = f(z_{j+1})  - \tilde{f}_{j}(z_{j+1}) \geq 0$.
\end{proof}
We next show that the error $\tilde{\epsilon}_{j+1}$ can be upper bounded by the decrease in function value whenever a descent step happens, i.e., the descent criterion \eqref{eq:testing-bundle} holds. The proof is provided in \cref{subsec:proof-inexact-ppm}.

\begin{prop} 
    \label{prop:inexact-ppm-interp}
    Let $f : \RR^n \to \RR$ be convex. The iterate $x_{k + 1} =$~\ProDes($y_k, \beta, \rho)$ with $\rho > 0$ and $\beta \in (0,1)$  satisfies 
    \begin{equation} \label{eq:inexact-iterate}
         x_{k+1} = y_k -g_{k + 1}/\rho, \; g_{k + 1} \in \partial_{\epsilon_{k+1}} f(x_{k+1}), \; \epsilon_{k + 1} 
        \leq \frac{1 - \beta}{\beta}\left(f(y_k) - f(x_{k + 1})\right).
    \end{equation}
\end{prop}
Note that we replaced $z_{j + 1}$ by $x_{k + 1}$ in \cref{prop:inexact-ppm-interp} since $z_{j + 1}$ satisfies the descent criterion \eqref{eq:testing-bundle} by assumption. \Cref{prop:inexact-ppm-interp} bounds the inexactness of the subgradient. It is not immediate, however, that the inexactness in \cref{prop:inexact-ppm-interp} satisfies the acceleration condition \cref{eq:A-HPE-condition}. Fortunately, under the smoothness assumption and with a correctly chosen $\beta$, we can guarantee the condition \cref{eq:A-HPE-condition}.

The key insight is to use smoothness to relate the function difference $f(y_k) - f(x_{k + 1})$ to the inexact subgradient interpretation in \cref{prop:inexact-subgrad}. Then, notice that we can choose $\beta$ so that the $\frac{1 - \beta}{\beta}$ term in \Cref{eq:inexact-iterate} contributes a damping factor to the bound. This damping mechanism enables the enlarged step size choice of \cref{alg:A-PBM-main} (i.e., we allow $\rho < M$) when compared to Nesterov's AGD that always requires a small step size $1/M$. We formalize this result in the next lemma. The proof is postponed to \cref{subsec:prop:smooth-epsilon-bound}. 

\begin{lem}
    \label{lemma:smooth-epsilon-bound}
    Let $f : \RR^n \to \RR$ be $M$-smooth and convex, and choose $\rho \geq \frac{M}{c}$, $\beta \in \big[ \frac{c + 2\sqrt{c} + 2}{c + 2\sqrt{c} + 3},1 \big)$ for $c > 0$ arbitrary. The iterate $x_{k + 1} =$~\ProDes($y_k, \beta, \rho)$ satisfies
    \begin{equation*}
       x_{k+1} = y_k - {g}_{k + 1}/\rho, \; {g}_{k + 1} \in \partial_{\epsilon_{k+1}} f(x_{k+1}), \; 2\epsilon_{k + 1}\leq \rho \|x_{k + 1} - y_k\|^2.
    \end{equation*}
\end{lem}

Combining \Cref{lemma:smooth-epsilon-bound} with \Cref{thm:monteiro-conv-main} confirms that the iterate $\{x_k\}$ in \cref{alg:A-PBM-main} satisfies 
$$ f(x_k) - f^\star \leq \frac{2 \rho \cdot \Dist(x_0, S)^2}{k^2}, \; \forall k\geq 1.$$ Finally, \cref{lemma:descent-step} ensures that each call of~\ProDes($y_k,\beta,\rho$) terminates in at most  $\frac{16(M+\rho)^3}{(1-\beta)^2 \rho^3}$ iterations. Combining these two results completes the proof of \cref{thm:A-PBM}.

\vspace{-3mm}
\section{Numerical experiments}
\label{sec:Numerical}

In this section, we perform two numerical experiments to test the numerical performance of our proposed \Cref{alg:A-PBM-main}. We use the essential model \cref{eq:finite-memory-models} to construct the under-estimator in the oracle \ProDes($y_k,\beta,\rho$) in \cref{alg:Proxi-descent-subproblem}. In this case, the subproblem in \cref{eq:PBM-trial-point} admits an analytical solution; see e.g., \citet[Claim 1]{diaz2023optimal}, or \citet[Appendix B.6]{liao2025proximal}. 

Our first experiment is to verify the $\bigO(1/k^2)$ convergence guaranteed by \cref{thm:A-PBM}. To demonstrate the worst convergence rate, we consider the function $f(x) = \frac{1}{8}x^\tr Lx - \frac{1}{4}\langle x, e_1\rangle,$ where $L \in \mathbb{R}^{200 \times 200}$ is the matrix which is $2$ on the diagonal and $-1$ on the off-diagonals, and $e_1$ is the standard basis vector of the first coordinate. This function $f$ is $1$-smooth and convex but not strongly convex. It is also used in \citet[Section 2.1.7]{nesterov2018lectures} to show the $\Omega(1/k^2)$ complexity bounds for smooth and convex functions under a first-order oracle. We run both the classical PBM and accelerated PBM for $1000$ iterations. The numerical result is presented in \cref{fig:worst}-(a). In \cref{fig:worst}-(a), we see that the classical PBM only has the usual $\bigO(1/k)$ convergence behavior. In contrast, the accelerated PBM enjoys a much faster $\bigO(1/k^2)$ convergence rate, validating our theoretical findings in \cref{thm:A-PBM}.

Our second experiment considers the logistic regression objective $f(w) = \frac{1}{m}\sum_{i = 1}^m \log(1 + e^{-y_i\langle x_i, w\rangle})$ with synthetic feature-label pairs $\{(x_i, y_i)\}_{i = 1}^m $ such that  $ x_i \in \RR^{200}$, $ y_i = \pm 1$, and $m = 200$. We generate the problem data such that the objective has a smoothness constant bounded by $1000$ and a minimizer $w^\star = 0$. We compare Nesterov’s AGD (\cref{alg:NGD}) with our proposed accelerated PBM (\Cref{alg:A-PBM-main}) with different choices of parameters. The result of the experiment is presented in \cref{fig:worst}-(b). We see that the accelerated PBM instances that use a larger step size (smaller $\rho$) greatly speed up the convergence. For example, the accelerated PBM with $\rho = 0.7$ and $\beta = 0.995$ achieves the accuracy of $10^{-14}$ in $20$ iterations, while the AGD only achieves the accuracy of $10^{-3}$ in $200$ iterations. This numerical experiment also validates our theoretical finding in \cref{thm:A-PBM} that the convergence is guaranteed with a larger step size.

\begin{figure}[h]
\centering
\setlength\textfloatsep{0pt}
\setlength{\belowcaptionskip}{0pt}
\setlength{\abovecaptionskip}{0pt}
\subfigure[Worst-case function]{\includegraphics[width=0.45\textwidth,height=0.20\textheight]{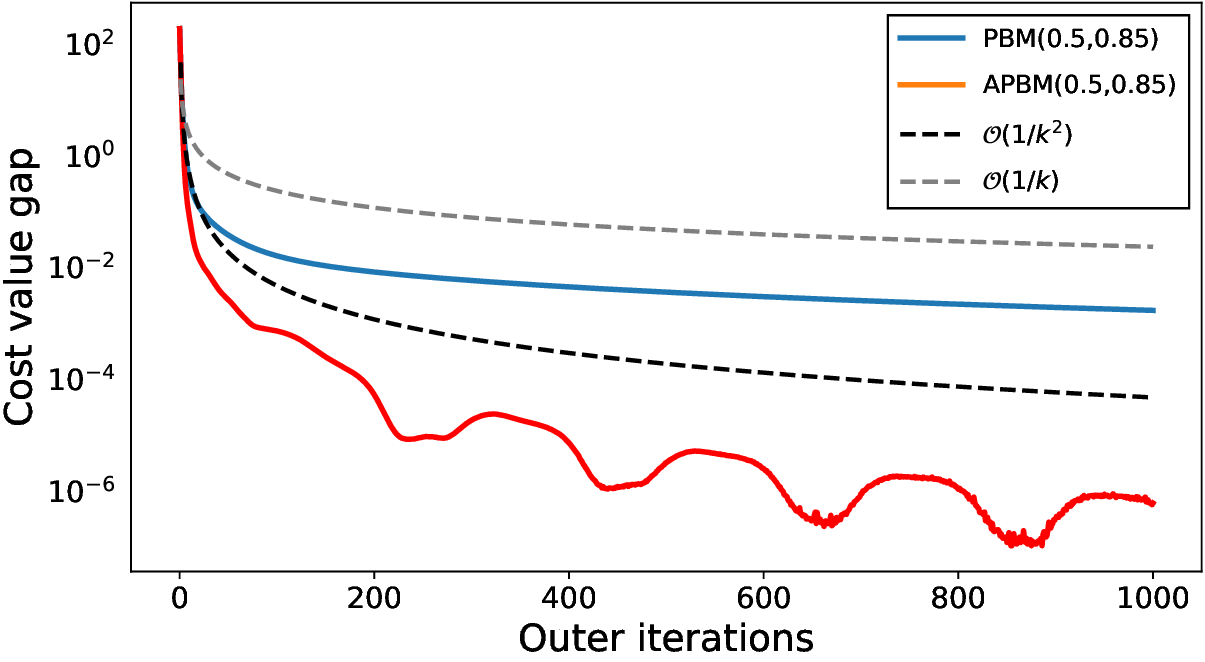}} 
\quad
\subfigure[Logistic regression]{\includegraphics[width=0.45\textwidth,height=0.20\textheight]{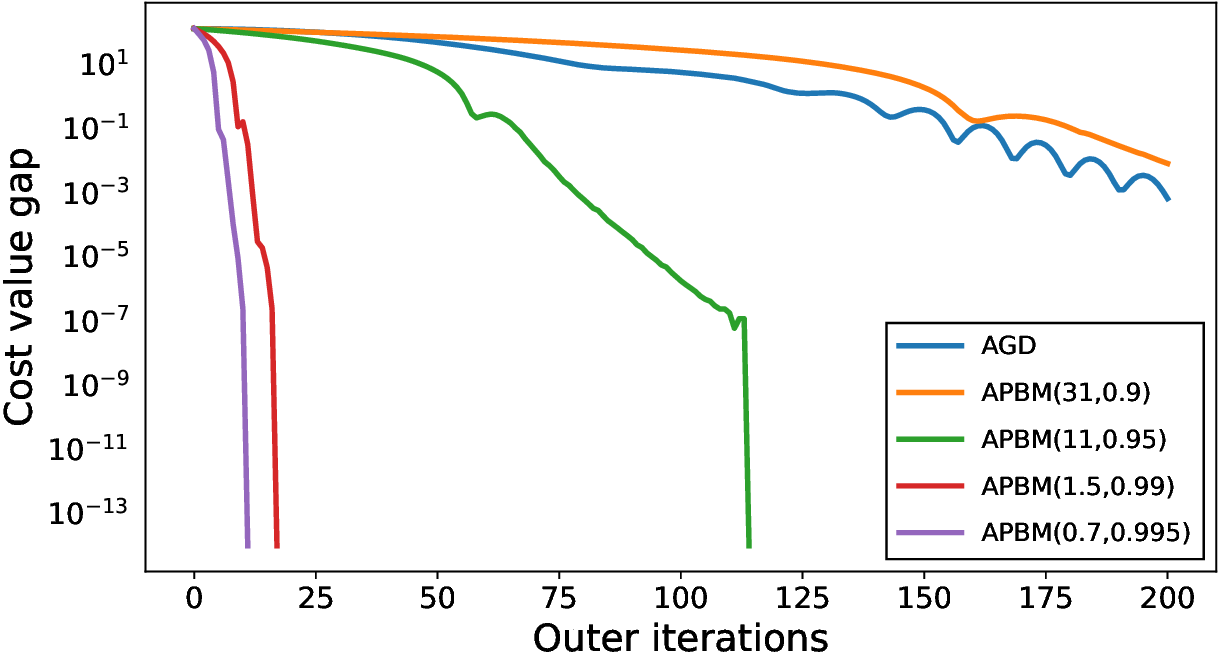}}
\caption{Numerical experiments. The worst-case function is taken from \cite{nesterov2018lectures}. The notation PBM$(x,y)$ and APBM$(x,y)$ denotes the PBM and the accelerated PBM with parameters $\rho = x$ and $\beta = y$.}
\label{fig:worst}
\vspace{-6 mm}
\end{figure}

\vspace{-1mm}
  \section{Conclusion}
\label{sec:conclusion}
\vspace{-1mm}
We have introduced a new accelerated proximal bundle method (PBM), which naturally integrates Nesterov's acceleration scheme with the classical PBM framework. The proposed method achieves the optimal $\mathcal{O}(1/\sqrt{\epsilon})$ convergence rate for smooth convex functions, and the theory is supported by numerical results. 
Since PBM is inherently designed for nonsmooth optimization, an interesting future direction is to extend these convergence guarantees beyond the smooth~setting.

\newpage
\acks{This work is supported by NSF ECCS-2154650, NSF CMMI 2320697, and NSF CAREER 2340713.}

\appendix

\bibliography{reference}

\newpage 

\section{Additional properties of the proximal bundle method}
\subsection{Proof of \cref{prop:connection-NGD}}
\label{subsec:connection-NGD-pf}
To prove \cref{prop:connection-NGD}, we first recall a result from \citep{liao2025bundle}.
\begin{theorem}[{\citet[Theorem 8]{liao2025bundle}}]
\label{thm:thm-8-lag}
    Let $\alpha > 0$, $y_k \in \mathbb{R}^n$, and $f : \mathbb{R}^n \to \mathbb{R}$. Suppose that $\tilde{f}:\mathbb{R}^n \to \mathbb{R}$ is a convex function that satisfies 
    \begin{equation}
        \tilde{f}(x) \leq f(x) \leq \tilde{f}(x) + \frac{\alpha}{2}\|x - y_k\|^2, \quad \forall \ x \in \mathbb{R}^n. \label{eq:model-quad-close}
    \end{equation}
    If $\theta \geq \alpha$ and $z = \argmin_x\left(\tilde{f}(x) + \frac{\theta}{2}\|x - y_k\|^2 \right)$, then the following holds:
    \begin{equation}
        0 \leq f(z) - \tilde{f}(z) \leq f(y_k) - f(z).
    \end{equation}
\end{theorem}
As \cref{prop:connection-NGD} considers a convex and $M$-smooth function $f$, it is clear that for any $y_k \in \RR^n$, the function $\tilde{f}(\cdot) = f(y_k) + \langle \nabla f(y_k), \cdot - y_k \rangle$ and the constant $\alpha = M$ satisfy \cref{eq:model-quad-close}. With this theorem in hand, we are ready to prove \cref{prop:connection-NGD}.\\

\noindent \textbf{Proof of \cref{prop:connection-NGD}:} 
Given a center point $y_k$, by \cref{assump:bm}, we know that $\tilde{f}_1(y) \leq f(y)$ for all $y$, and 
\begin{align*}
    &f(y_k) + \innerproduct{\nabla f(y_k )}{y-y_k} \leq \tilde{f}_{1}(y),\; \forall y \in \RR^n \\
    \Longrightarrow \; & f(y) \leq  f(y_k) + \innerproduct{\nabla f(y_k )}{y-y_k} + \frac{\rho}{2}\|y- y_k\|^2 \leq \tilde{f}_{1}(y) + \frac{\rho}{2}\|y- y_k\|^2 ,\; \forall y \in \RR^n, \rho \geq M,
\end{align*}
where the implication is due to the $M$-smoothness. Thus, we see that $\tilde{f}_1 $ satisfies \cref{eq:model-quad-close}. Invoking \cref{thm:thm-8-lag} on the update $z_2 = \argmin_{y} \{ \tilde{f}_1(y) + \frac{\rho}{2}\|y - y_k\|^2 \} $, we obtain
\begin{equation}
    \label{eq:conseq}
    0 \leq f(z_2) - \tilde{f}(z_2) \leq f(y_k) - f(z_2).
\end{equation}
Finally, we calculate that
\begin{align}
        \beta(f(y_k) - f_k(z_{2})) & = \beta(f(y_k) - f(z_{2}) + f(z_{2}) - f_k(z_{2})) \label{eq:add-sub}\\
        & \leq \beta(f(y_k) - f(z_{2}) + f(y_k) - f(z_{2})) \label{eq:thm-8-bound}\\
        & = 2\beta(f(y_k) - f(z_{2}) \nonumber \\
        & \leq f(y_k) - f(z_{2}) \label{eq:beta-1/2}
    \end{align}
    where we added and subtracted $f(z_{2})$ within the parentheses in \Cref{eq:add-sub}, applied \Cref{eq:conseq} in \Cref{eq:thm-8-bound}, and used that $\beta \leq \frac{1}{2}$ in \Cref{eq:beta-1/2}. Thus, the acceptance criterion \Cref{eq:testing-bundle} is satisfied for $z_{2}$, so \cref{alg:Proxi-descent-subproblem} terminates in one iteration. \hfill $\square$

\subsection{Proof of \cref{prop:inexact-ppm-interp}}
\label{subsec:proof-inexact-ppm}

For completeness, we prove \cref{prop:inexact-ppm-interp} now, a result which allows us to bound differences in function values as opposed to bounding the model gap.\\

\noindent \textbf{Proof of \cref{prop:inexact-ppm-interp}:} 
From \cref{prop:inexact-subgrad}, we can write $x_{k+1} =$~\ProDes($y_k, \beta, \rho$) as $x_{k+1} = y_k - \frac{1}{\rho} \tilde{g}_J$ where $J$ is the last iteration index in \cref{alg:Proxi-descent-subproblem} and 
\begin{align*}
   \tilde{g}_J \in \partial f_{\tilde{\epsilon}_{J+1}}(x_{k+1}), \; \tilde{\epsilon}_{J+1} = f(x_{k + 1}) - \tilde{f}_J(x_{k + 1}).
\end{align*}
Moreover, the error $ f(x_{k + 1}) - \tilde{f}_J(x_{k + 1})$ can be further bounded as
\begin{align}
    f(x_{k + 1}) - \tilde{f}_J(x_{k + 1}) & \leq f(y_k) - \beta\big(f(y_k) - \tilde{f}_J(x_{k + 1})\big) - \tilde{f}_J(x_{k + 1}) \label{eq:bundle-cons-1}\\
    & = (1 - \beta)\big(f(y_k) - \tilde{f}_J(x_{k + 1})\big) \label{eq:bundle-combine}\\
    & \leq \frac{1 - \beta}{\beta}\big(f(y_k) - f(x_{k + 1})\big), \label{eq:bundle-cons-2}
\end{align}
where \eqref{eq:bundle-cons-1} and \eqref{eq:bundle-cons-2} apply the definition of the test in \eqref{eq:testing-bundle}, and \eqref{eq:bundle-combine} combines like terms. Finally, letting $g_{k + 1} = \tilde{g}_J$ and $ \epsilon_{k+1} =  \tilde{\epsilon}_{J+1}$ completest the proof.

\hfill $\square$ 

\subsection{Proof of \cref{lemma:smooth-epsilon-bound}}
\label{subsec:prop:smooth-epsilon-bound}
An important ingredient in the proof of \cref{lemma:smooth-epsilon-bound} is the fact that, given a smooth function, the distance between a true gradient and an inexact subgradient can be controlled by the inexactness. This is formalized in the following lemma.
\begin{lem}
    \label{lem:inexact-subgrad-closeness}
    Let $f : \RR^n \to \RR$ be $M$-smooth and convex. Suppose that $g \in \partial_\epsilon f(x)$ with an arbitrary $\epsilon > 0$. Then, we have
    \begin{equation*}
        \|\nabla f(x) - g\| \leq \sqrt{2M \epsilon}.
    \end{equation*}
\end{lem}
\begin{proof}
    By definition,  $g \in \partial_\epsilon f (x)$ ensures that the following lower bound holds for all $y \in \RR^n$:
    \begin{equation}
        f(y) \geq f(x) + \langle g, x - y\rangle - \epsilon. \label{eq:inexact-subgrad-def}
    \end{equation}
    Specifying \eqref{eq:inexact-subgrad-def} to the choice $y = x - \frac{1}{M}(\nabla f(x) - g)$ gives the lower bound
    \begin{equation}
        f(x - \frac{1}{M}(\nabla f(x) - g)) \geq f(x) - \frac{1}{M} \langle g, \nabla f(x) - g\rangle - \epsilon. \label{eq:inexact-subgrad-lb}
    \end{equation}
    Meanwhile, smoothness gives an upper bound on $f(x - \frac{1}{M}(\nabla f(x) - g))$ via
    \begin{equation}
        f(x - \frac{1}{M}(\nabla f(x) - g)) \leq f(x) - \frac{1}{M}\langle \nabla f(x), \nabla f(x) - g\rangle + \frac{M}{2}\|-\frac{1}{M}(\nabla f(x) - g)\|^2. \label{eq:inexact-subgrad-ub}
    \end{equation}
    Therefore, combining the bounds in \eqref{eq:inexact-subgrad-lb} and \eqref{eq:inexact-subgrad-ub} gives
    \begin{align}
        -\frac{1}{M}\langle g, \nabla f(x) - g\rangle - \epsilon & \leq - \frac{1}{M}\langle \nabla f(x), \nabla f(x) - g\rangle + \frac{1}{2M}\|\nabla f(x) - g\|^2 \notag\\
        \Rightarrow \quad \frac{1}{M}\langle \nabla f(x) - g, \nabla f(x) - g\rangle - \epsilon & \leq \frac{1}{2M}\|\nabla f(x) -g\|^2 \label{eq:inexact-subgrad-rearr1}\\
        \Rightarrow \quad \|\nabla f(x) - g\|^2 \leq 2M\epsilon \label{eq:inexact-subgrad-rearr2}
    \end{align}
    where \eqref{eq:inexact-subgrad-rearr1} and \eqref{eq:inexact-subgrad-rearr2} rearrange terms. Taking square roots on both sides of \eqref{eq:inexact-subgrad-rearr2} gives the claim.
\end{proof}
With \cref{lem:inexact-subgrad-closeness}, we now prove \cref{lemma:smooth-epsilon-bound}. The idea is to bound the backwards difference $f(y_k) -f(x_{k + 1})$ using smoothness, which is possible due to the inexact PPM interpretation.\\

\noindent \textbf{Proof of \cref{lemma:smooth-epsilon-bound}:} 
By \cref{prop:inexact-ppm-interp}, it suffices to bound the quantity
\begin{equation*}
    \hat{\epsilon} := \frac{1 - \beta}{\beta}\big(f(y_k) - f(x_{k + 1})\big)
\end{equation*} 
 by $\frac{\rho}{2}\|x_{k + 1} - y_k\|^2$. We begin to bound $\hat{\epsilon}$ as follows:
 \begin{subequations}
\begin{align}
    \hat{\epsilon} & \leq \frac{1 - \beta}{\beta}\left(\langle \nabla f(x_{k + 1}), y_k - x_{k + 1}\rangle + \frac{M}{2}\|x_{k + 1} - y_k\|^2 \right) \label{eq:prop4-smoothness}\\
    & \leq \frac{1 - \beta}{\beta}\left(\langle \nabla f(x_{k + 1}) - g_{k + 1}, y_k - x_{k + 1}\rangle + \langle g_{k + 1}, y_k - x_{k + 1} \rangle + \frac{M}{2}\|x_{k + 1} - y_k\|^2\right) \label{eq:prop4-add-subtract}\\
    & = \frac{1 - \beta}{\beta}\left(\langle \nabla f(x_{k + 1}) - g_{k + 1}, y_k - x_{k + 1}\rangle + \rho\|x_{k + 1} - y_k\|^2 + \frac{M}{2}\|x_{k + 1} - y_k\|^2 \right) \label{eq:prop4-v-def},
\end{align}
 \end{subequations}
where \Cref{eq:prop4-smoothness} applies $M$-smoothness of the objective function, \cref{eq:prop4-add-subtract} adds and subtracts $g_{k + 1}$ within the inner product, and \cref{eq:prop4-v-def} uses that $g_{k + 1} = \rho(y_k - x_{k + 1})$ as was shown in \cref{prop:inexact-ppm-interp}. Bounding the inner product above and applying \cref{lem:inexact-subgrad-closeness} with $x = x_{k + 1}$ yields the bound
\begin{align}
    \hat{\epsilon} & \leq \frac{1 - \beta}{\beta}\left(\sqrt{2M\hat{\epsilon}}\|x_{k + 1} - y_k\| + \big(\rho + \frac{M}{2}\big)\|x_{k + 1} - y_k\|^2 \right) \label{eq:prop4-lemma-app}.
\end{align}
We see that \eqref{eq:prop4-lemma-app} is a quadratic inequality in $\sqrt{\hat{\epsilon}}$. As the coefficient of the quadratic term is positive, an upper bound on $\hat{\epsilon}$ is given by the square of the largest solution to the quadratic. We deduce
\begin{align}
    \hat{\epsilon} \leq \left(\frac{1 - \beta}{\beta}\right)^2\left(M + \frac{\beta}{1 - \beta}\big(\rho + \frac{M}{2}\big) + \sqrt{M^2 + \frac{2M\beta}{1 - \beta}\big(\rho + \frac{M}{2}\big)}\right)\|x_{k + 1} - y_k\|^2. \label{eq:prop4-final-bound}
\end{align}
We now finish the proof of \cref{lemma:smooth-epsilon-bound} by establishing the fact that whenever 

\begin{subequations} \label{eq:smoothness-parameter-bound}
\begin{equation} \label{eq:smoothness-parameter-bound-a}
    \rho \geq \frac{M}{c}, \quad \beta \in \left[\frac{c + 2\sqrt{c} + 2}{c + 2\sqrt{c} + 3}, 1\right)
\end{equation}
we have 
\begin{equation} \label{eq:smoothness-parameter-bound-b}
    \left(\frac{1 - \beta}{\beta}\right)^2\left(M + \frac{\beta}{1 - \beta}\big(\rho + \frac{M}{2}\big) + \sqrt{M^2 + \frac{2M\beta}{1 - \beta}\big(\rho + \frac{M}{2}\big)}\right) \leq \frac{\rho}{2}. 
\end{equation}
\end{subequations}
This inequality \Cref{eq:smoothness-parameter-bound} is not difficult to show, only requiring single variable calculus. For completeness, we provide the details in \cref{appendix:smoothness-bound}.

\subsection{The inequality in \Cref{eq:smoothness-parameter-bound}} \label{appendix:smoothness-bound}
Here, we show that the parameters in \cref{eq:smoothness-parameter-bound-a} always guarantee the bound in \cref{eq:smoothness-parameter-bound-b}.  To simplify notation, we write
\begin{equation}
    b := \frac{1 - \beta}{\beta}, \quad h_b(\rho) := b^2\left(M + \frac{1}{b}\big(\rho + \frac{M}{2}\big) + \sqrt{M^2 + \frac{2M}{b}\big(\rho + \frac{M}{2}\big)}\right)  \label{eq:param-h-def}
\end{equation}
and our goal becomes showing that $h_b(\rho) \leq \frac{\rho}{2}$ when \cref{eq:smoothness-parameter-bound-a} holds. To do this, we will first show that our restriction on $\beta \in [\frac{c + 2\sqrt{c} + 2}{c + 2\sqrt{c} + 3}, 1)$ (corresponding to $b \in (0,\frac{1}{c + 2\sqrt{c} + 2}]$) ensures that, keeping $\beta$ fixed, $h_b'(\rho) \leq \frac{1}{2}$ for $\rho \geq \frac{M}{c}$. Then, we will show that we have $h_b(\rho) = \frac{\rho}{2}$ for the smallest choices of $\rho$ and the largest choice of $b$ (corresponding to the smallest choice of $\beta$ since $b = \frac{1-\beta}{\beta}$ increases as $\beta$ decreases). Finally, since $\left(\frac{\rho}{2}\right)' = \frac{1}{2}$, and $h_b(\rho)$ is simultaneously increasing in $b$ and $\rho$, we conclude that $h_b(\rho) \leq \frac{\rho}{2}$ for all  $\rho \geq \frac{M}{c}, 
b \in (0,\frac{1}{c + 2\sqrt{c} + 2}]$.\\

\noindent \textbf{Step 1.} Show $h_b'(\rho) \leq \frac{1}{2}$ for all $\rho \geq \frac{M}{c}$ and $b \in (0,\frac{1}{c + 2\sqrt{c} + 2}]$:
Taking derivatives of $h_b(\rho)$ with respective to $\rho$ gives that 
\begin{equation}
    h_b'(\rho) = b + b^2\frac{1}{2\sqrt{M^2 + \frac{2M}{b}(\rho + \frac{M}{2}\big)}}\frac{2M}{b} = b + \frac{Mb}{\sqrt{M^2 + \frac{2M}{b}(\rho + \frac{M}{2})}} \label{eq:param-h'-initial}.
\end{equation}
Since $\rho + \frac{M}{2} \geq \rho$ and $\rho \geq \frac{M}{c}$, we can write
\begin{equation*}
    M^2 + \frac{2M}{b}\big(\rho + \frac{M}{2}\big) \geq M^2 + \frac{2M}{b}\rho \geq M^2 + \frac{2M^2}{bc}
\end{equation*}
so that a specialized upper bound on $h_b'(\rho)$ is given by
\begin{equation}
   \tilde{h}(b) :=  b + \frac{Mb}{\sqrt{M^2 + \frac{2M^2}{bc}}} = b + \frac{b}{\sqrt{1 + \frac{2}{bc}}}.
\end{equation}
We shift our attention to obtaining $\tilde{h}(b) \leq \frac{1}{2}$ from which the claim follows. Observe that $\tilde{h}(b)$ is increasing in $b$. The range of $\beta \in [\frac{c + 2\sqrt{c} + 2}{c + 2\sqrt{c} + 3},1)$ corresponds to the range $b \in (0,\frac{1}{c + 2\sqrt{c} + 2}]$.  We aim to show the claim at the extreme choice $b_{\text{max}} :=  \frac{1}{c + 2\sqrt{c} + 2}$.
Direct computation gives 
\begin{subequations}
\begin{align}
    \tilde{h}(b_{\text{max}}) & = \frac{1}{c + 2\sqrt{c} + 2}\left(1 + \frac{\sqrt{c}}{\sqrt{3c + 4\sqrt{c} + 4}}\right) \notag\\
    & \leq \frac{1}{c + 2\sqrt{c} + 2}\left(1 + \frac{\sqrt{c}}{2 + \sqrt{c}}\right) \label{eq:param-h-tilde-bound-1}\\
    & = \frac{2(\sqrt{c} + 1)}{(c + 2\sqrt{c} + 2)(\sqrt{c} + 2)} \notag\\
    & \leq \frac{2(\sqrt{c} + 1)}{4(\sqrt{c} + 1)} = \frac{1}{2}, \label{eq:param-h-tilde-bound-2}
\end{align}
\end{subequations}
where the equations \cref{eq:param-h-tilde-bound-1} and \cref{eq:param-h-tilde-bound-2} apply the two elementary inequalities, respectively:
\begin{itemize}
    \item $(\sqrt{c} + 2)^2 = c + 4\sqrt{c} + 4 \leq 3c + 4\sqrt{c} + 4$
    \item $4(\sqrt{c} + 1) \leq c^{3/2} + 4c + 6\sqrt{c} + 4 = (\sqrt{c} + 2\sqrt{c} + 2)(\sqrt{c} + 2)$.
\end{itemize}

\noindent \textbf{Step 2.} Show that the choice $\rho = \frac{M}{c},b_{\text{max}} = \frac{1}{c + 2\sqrt{c} + 2}$ obtains equality, i.e. $h_{b_{\text{max}}}\big(\frac{M}{c}\big) = \frac{\rho}{2}$: 

Under this choice of parameters, we have
\begin{equation}
    M = c\rho, \qquad \text{so that} \qquad M + \frac{\rho}{2} = \rho\big(1 + \frac{c}{2}\big) \label{eq:param-ext-simp},
\end{equation}
and thus the expression within the radical of $h_{b_{\text{max}}}\big(\frac{M}{c}\big)$ becomes
\begin{subequations}
\begin{align}
    M^2 + 2M(c + 2\sqrt{c} + 2)(\rho + \frac{M}{2}\big) & = c^2\rho^2 + 2c\rho^2(c + 2\sqrt{c} + 2)\big(1 + \frac{c^2}{2}\big) \label{eq:param-ext-comp}\\
    & = \rho^2\left(c^2 + 2c\big(c +  2\sqrt{c} + 2\big)\big(1 + \frac{c^2}{2}\big)\right) \label{eq:param-ext-fact-1}\\
    & = \rho^2\left(c^2 + c\big(c^2 + 2c^{3/2} + 4c + 4c^{1/2} + 4\big)\right) \label{eq:param-ext-mult}\\
    & = \rho^2c(c + c^{1/2} + 2)^2 \label{eq:param-ext-square-1}
\end{align}
\end{subequations}
where \cref{eq:param-ext-comp} applies \cref{eq:param-ext-simp} and our choice of $\beta$, \cref{eq:param-ext-fact-1} factors $\rho^2$ out of \eqref{eq:param-ext-comp}, \cref{eq:param-ext-mult} multiplies out the expressions within the parentheses, and \cref{eq:param-ext-square-1} notices that the expression in \eqref{eq:param-ext-mult} a square. We obtain
\begin{subequations}
\begin{align}
    h_{b_{\text{max}}}\big(\frac{M}{c}\big) & = b^2_{\text{max}}\left(c\rho + \frac{1}{b_{\text{max}}}\rho\big(1 + \frac{c^2}{2}\big) + \rho c^{1/2}\big(c + c^{1/2} + 2\big)\right) \label{eq:param-ext-app-rad}\\
    & = \rho b_{\text{max}}^2\left(c + \big(c + 2\sqrt{c} + 2\big)\big(1 + \frac{c^2}{2}\big) + c^{1/2}\big(c + c^{1/2} + 2\big)\right) \label{eq:param-ext-rho-fact-2}\\
    & = \rho b_{\text{max}}^2\frac{(c + 2\sqrt{c} + 2)^2}{2} \label{eq:param-ext-square-2}\\
    & = \frac{\rho}{2} \notag,
\end{align}
\end{subequations}
where \cref{eq:param-ext-app-rad} substitutes \cref{eq:param-ext-simp} and \cref{eq:param-ext-square-1} into the definition of $h_{b_{\text{max}}}\big(\frac{M}{c}\big)$, \cref{eq:param-ext-rho-fact-2} factors $\rho$ out, and \cref{eq:param-ext-square-2} recognizes a square once more.

\vspace{1mm}
\noindent \textbf{Step 3.} Finally, we note that $h_{b}(\rho)$ increases when $b$ increases or $\rho$ increases. From step 2, we know that the largest choice of $b = b_{\text{max}}$ and the smallest choice of $\rho = \frac{M}{c}$ gives $h_{b_{\text{max}}}\big(\rho \big) = \frac{\rho}{2}$. To conclude, we only need to argue that for all $\rho \geq \frac{M}{c}$, we have $h_{b_{\text{max}}}\big(\rho \big) \leq  \frac{\rho}{2}$. This is indeed the case due to step 1. In other words, step 1 tells us $h_b'(\rho) \leq \frac{1}{2}$ for all $\rho \geq \frac{M}{c}$ and $b \in (0,b_{\text{max}}]$, and we know $(\frac{\rho}{2})' = \frac{1}{2}$. Since  $h_{b_{\text{max}}}\big(\rho \big) = \frac{\rho}{2}$ when $\rho = \frac{M}{c}$, we conclude that $h_{b_{\text{max}}}(\rho)$ remains bounded by $\frac{\rho}{2}$ as $\rho$ increases from $\frac{M}{c}$.
\hfill $\square$

\subsection{Equivalence between the double-loop \cref{alg:bundle} and classical single-loop PBM}
\label{subsec:reinterpretation}
This subsection discusses the equivalence between \cref{alg:bundle} and the classical single-loop PBM. In other words, we show that \cref{alg:bundle} is a reformulation of the classical PBM. Similar discussion can also be found in \citet[Section 4.3]{liao2025bundle}. Let us first review the classical single-loop PBM in \cref{alg:PBM-classical}.
At every iteration $k$, it solves the following subproblem
\begin{align*}
    z_{k+1}= \argmin_{y}  \left\{f_k(y) + \frac{\rho}{2}\|y - y_k\|^2 \right\}
\end{align*}
to get candidate solution $z_{k+1}$. To decide if $z_{k+1}$ makes sufficient descent, we adapt the test
\begin{align}
    \label{eq:test}
    \beta (f(y_k) - f_k(z_{k+1})) \leq f(y_k) -  f(z_{k+1}), 
\end{align}
where $\beta \in (0,1)$ is a pre-defined constant. If \cref{eq:test} holds, then we set $y_{k+1} = z_{k+1}$ (known as a descent step). Otherwise, we set $y_{k+1} = y_k$ (this is called a null step). Afterwards, regardless of a descent step or a null step, the algorithm updates the model $f_{k+1}$ following \cref{assumption:approximation-conditions}.

Below, we explain the equivalence between \cref{alg:PBM-classical,alg:bundle}. First, the test \cref{eq:test} is the same as the test \cref{eq:testing-bundle}. Second, the under-estimator $\tilde{f}_j$ in \cref{alg:Proxi-descent-subproblem} of \cref{alg:bundle} is equivalent to the $f_k$ in \cref{alg:PBM-classical}. The $\tilde{f}_j$ in \cref{alg:PBM-classical} resets the index $j$ whenever a new center point is acquired, whereas \cref{alg:PBM-classical} keeps the iteration count $k$ throughout the whole process. Therefore, the only difference is the notation and indices. One can view the oracle \ProDes($y_k,\beta,\rho$) in \cref{alg:bundle} as a cycle of null steps in \cref{alg:PBM-classical} and the satisfaction of the stopping criterion in \ProDes($y_k,\beta,\rho$) corresponds to the fulfillment of
the test \cref{eq:test}. Third, the assumptions for constructing the under-estimators in both \cref{alg:PBM-classical,alg:bundle} are the same, i.e., \cref{assump:bm,assumption:approximation-conditions} are the same.  

In summary, \cref{alg:bundle} is a re-interpretation of the classical single-loop PBM \cref{alg:PBM-classical}. \cref{alg:bundle}, however, provides a cleaner distinction between null and descent steps.

\begin{algorithm}[t]
\caption{Classical proximal bundle method}
\label{alg:PBM-classical}
\begin{algorithmic}
\REQUIRE $y_1 \in \RR^n, T > 0, \rho > 0,$ $\beta \in (0,1)$
\FOR{$k=1,2, \ldots, T$}
    \STATE Compute 
    $
         z_{k+1}= \argmin_{y} \left\{ f_k(y) + \frac{\rho}{2}\|y - y_k\|^2 \right\};
    $
    \IF {$\beta (f(y_k) - f_k(z_{k+1})) \leq f(y_k) -  f(z_{k+1})  $ }
        \STATE Set $y_{k+1} = z_{k+1}$; \hfill \textit{Descent step}
    \ELSE
        \STATE Set $y_{k+1} = y_k$; \hfill \textit{Null step}
    \ENDIF
    \STATE Construct $f_{k+1}$ that approximates $f(\cdot)$ satisfying \cref{assumption:approximation-conditions}
\ENDFOR
\end{algorithmic}
\end{algorithm}

\begin{assumption} \label{assumption:approximation-conditions}
The convex function $f_{k+1}$ satisfies three conditions:
\begin{enumerate}
\setlength{\itemsep}{0pt}
    \item \textbf{Lower approximation:} Global convex lower approximation, 
    $
        f_{k+1}(y) \leq f(y), \forall y \in \mathbb{R}^n.
    $
    \item  \textbf{Subgradient lower bound:} We have  
    $
         f_{k+1}(y) \geq  f(z_{k+1})  + \innerproduct{g_{k+1}}{y-z_{k+1}},\forall y \in \mathbb{R}^n,
    $
    where $g_{k+1}$ satisfies $g_{k+1}  \in  \partial f(z_{k+1})$.  
        \item \textbf{Aggregation from the past approximation:} If  fails, then we require
        $
            f_{k+1}(y) \geq f_k(z_{k+1}) + \innerproduct{s_{k+1}}{y-z_{k+1}},\forall y \in \mathbb{R}^n,
        $
        where $s_{k+1} = \rho(y_k - z_{k+1}) \in \partial f_{k}(z_{k+1})$.
\end{enumerate} 
\end{assumption}

\section{Convergence of \cref{alg:A-HPE-simplified}}
\label{sec:A-HPE-simplified-conv}
Here, we prove the convergence of \cref{alg:A-HPE-simplified}, i.e., \cref{thm:monteiro-conv-main}. Our proof in this section largely follows \citet[Section 3]{monteiro2013accelerated}. We provide the proof details for self-containment. For convenience, we restate \cref{thm:monteiro-conv-main} below.

\setcounter{thm}{1}

\begin{thm}{\citet[Theorem 3.8]{monteiro2013accelerated}} \label{thm:monteiro-conv-apx}
    Let $f:\RR^n \to \RR$ be a convex function and set $S =  \argmin_{x} f(x)$. Assume $S \neq \emptyset$. The sequence $\{x_k\}$ in \cref{alg:A-HPE-simplified} satisfies
    \begin{equation*}
        f(x_k) - \min_{x} f(x) \leq \frac{2 \rho \cdot \Dist(x_0, S)^2}{k^2}, \; \forall k\geq 1.
    \end{equation*}
\end{thm}

Before providing a detailed proof, we first present an outline of the proof, which singles out the key inequalities. 

\noindent \textbf{Outline of the proof.} The core of the proof is to establish the following two key inequalities
\begin{subequations}
    \begin{align}
    A_k & \geq  \frac{k^2}{4} ,\;\forall k \geq 1, \label{eq:key-1} \\
        A_kf(x_k) + \frac{\rho}{2}\|x - z_k\|^2 & \leq A_kf(x) + \frac{\rho}{2}\|x - z_0\|^2, \quad \forall x \in \mathbb{R}^n. \label{eq:key-2}
    \end{align}
\end{subequations}
Plugging $x = \argmin_{y \in S} \|y - z_0\|$ into \cref{eq:key-2} and using \cref{eq:key-1} yields the desired result 
\begin{align*}
    f(x_k) - \min_{x} f(x)  \leq \frac{\rho \Dist(z_0,S)^2}{2 A_k } \leq \frac{2\rho \Dist(z_0,S)^2}{k^2}, \; \forall k \geq 1.
\end{align*} \hfill $\square$

In the next two subsections, we prove the two key inequalities \cref{eq:key-1,eq:key-2}.

\subsection{Proof of \cref{eq:key-1}}
Recall that the update of $a_{k}$ gives
    \begin{equation}
        a_{k } = \frac{1}{2} + \frac{\sqrt{1 + 4A_k}}{2}  \geq \frac{1}{2} + \sqrt{A_k}, \label{eq:a-k-bound}
    \end{equation}
The update $A_{k + 1} = A_k + a_{k}$ along with \Cref{eq:a-k-bound} gives
    \begin{align}
        A_{k + 1} & \geq A_k + \frac{1}{2} +  \sqrt{A_k} \geq A_k + \sqrt{A_k}+ \frac{1}{4}   = \left(\sqrt{A_k} + \frac{1}{2}\right)^2 \notag \\
        \Longrightarrow \quad \sqrt{A_{k + 1}} & \geq \sqrt{A_k} + \frac{1}{2}, \label{eq:A-k-recur}
    \end{align}
where we took squareroots in \Cref{eq:A-k-recur}. Adding \Cref{eq:A-k-recur} from $i = 0$ to $ k - 1$ with $A_0 = 0$ shows
    \begin{equation*}
    \sqrt{A_k} \geq \frac{k}{2} \quad \Rightarrow \quad A_k \geq \frac{k^2}{4}.
    \end{equation*} \hfill $\square$

\vspace{1 mm}

\subsection{Proof of \cref{eq:key-2}}
For notational convenience, given $k \geq 0$, we define the following functions
\begin{align}
    \gamma_{k+1}(x) := f(x_{k+1}) + \langle g_{k + 1}, x -x_{k+1} \rangle - \epsilon_{k+1},  \\
    \Gamma_0 \equiv 0, \quad \Gamma_{k + 1} = \frac{A_k}{A_{k + 1}}\Gamma_k + \frac{a_{k }}{A_{k + 1}}\gamma_{k + 1}, \label{eq:Gamma-k-def}
\end{align}
as well as the quantity
\begin{equation}
    \beta_k = \inf_{z \in \mathbb{R}^n}\left(A_k\Gamma_k(z) + \frac{\rho}{2}\|z - z_0\|^2\right) - A_kf(x_k). \label{eq:beta-k-def}
\end{equation}
The function $\gamma_{k + 1}$ is nothing but an under-estimator of $f$, i.e., $\gamma_{k + 1} \leq f$, as  $g_{k + 1} \in \partial_{\epsilon_{k+1}}f(x_{k+1})$. The function $\Gamma_{k+1}$ is a convex combination of $\Gamma_{k}$ and $\gamma_{k+1}$. The quantity $\beta_k$ will serve as an important element in the proof.

Below, we establish some useful properties for $\gamma_k$ and $\Gamma_{k}$.

\begin{lem}[{\citet[Lemma 3.2]{monteiro2013accelerated}}] \label{lemma:affine-maps} For integers $k \geq 0$, the following hold:
    \begin{enumerate}[label=(\alph*)] 
    \item $\gamma_{k + 1}$ is affine and $\gamma_{k + 1} \leq f$,
   \item $\Gamma_k$ is affine and $A_k\Gamma_k \leq A_kf$,
   \item $z_k = \argmin_{z \in \mathbb{R}^n} \{ A_k\Gamma_k(z) + \frac{\rho}{2}\|z - z_0\|^2\}$,
   \item $A_k\Gamma_k(x)  + \frac{\rho}{2}\|x - z_0\|^2 = A_k\Gamma_k(z_k)  + \frac{\rho}{2}\|z_k - z_0\|^2 + \frac{\rho}{2}\|x - z_k\|^2, \forall x \in \RR^n$.
    \end{enumerate}
\end{lem}
\begin{proof}
    \begin{enumerate}[label=(\alph*)] 
    \item For $k \geq 0$, $\gamma_{k + 1}$ is affine as it is the sum of a linear function and a constant, and $\gamma_{k + 1} \leq f$ as the definition 
    $ \gamma_{k+1}(x) = f(x_{k+1}) + \langle g_{k + 1}, x -x_{k+1} \rangle - \epsilon_{k+1}$ and 
    $g_{k +1} \in \partial_{\epsilon_{k+1}}f(x_{k+1})$. 

    \item As $\Gamma_k$ is a linear combination of $\{\gamma_i\}_{i = 1}^k$ for $k \geq 1$, we know that $\Gamma_k$ is affine by (a). We now show that $A_k\Gamma_k \leq A_kf$ by induction.
    \begin{itemize}
        \item Note that $\Gamma_0 = 0$ and $A_0 = 0$, so $A_k\Gamma_k \leq A_kf$ for $k = 0$ holds trivially.
        \item Suppose that $A_k\Gamma_k \leq A_kf$ for some $k > 0$. By the update \Cref{eq:Gamma-k-def}, we have that
    \begin{align*}
        A_{k + 1}\Gamma_{k + 1} & = A_k\Gamma_k + a_{k }\gamma_{k + 1}\\
        & \leq A_kf + a_{k}f = A_{k + 1}f
    \end{align*}
    where the inequality follows from the induction hypothesis and (a), and the last equality uses the update $A_{k+1} = A_k + a_k$.    
    \end{itemize}
     \item From the update of $z_k$ and assumption $z_0 = x_0$, we can rewrite $z_k$ as
     \begin{align}
     \label{eq:unwrap}
         z_k = z_0 - \sum_{i = 1}^k a_{i-1}g_{i-1}/\rho.
     \end{align}
     The claim would be established if we could show 
    \begin{equation}
    A_k \nabla \Gamma_k(x) = \sum_{i =1 }^{k} a_{i-1} g_{i-1} \quad \text{for }k \geq 1,  \label{eq:A-k-grad}
    \end{equation}
    since \cref{eq:unwrap,eq:A-k-grad} lead to 
    \begin{align*}
        z_k = x_0 - A_k \nabla \Gamma_k (z_k)/\rho,
    \end{align*}
    which is exactly the optimality condition of $z_k = \argmin_{z \in \mathbb{R}^n} \{ A_k\Gamma_k(z) + \frac{\rho}{2}\|z - z_0\|^2\}$.
    
    Below, we establish \cref{eq:A-k-grad} by induction.
    \begin{itemize}
        \item Notice that $a_0 = A_1 = 1$ and $\nabla \Gamma_1(x) = g_0$, so  \Cref{eq:A-k-grad} holds for $k = 1$.
        \item  Assume that \Cref{eq:A-k-grad} holds for some $k > 1$. Invoking \Cref{eq:Gamma-k-def} once more, we have that
    \begin{align*}
        A_{k + 1}\nabla \Gamma_{k + 1} & = A_k\nabla \Gamma_k  + a_{k}\nabla \gamma_{k + 1}\\
        & = \sum_{i = 1}^ka_{i-1}g_{i-1} + a_{k}g_{k } = \sum_{i = 1}^{k + 1}a_{i-1}g_{i-1}.
    \end{align*}
    
    \end{itemize}
    This proves (c).
    \item  By (b), $\Gamma_k$ is affine so there is a decomposition
    \begin{equation}
        A_k \Gamma_k(x) + \frac{\rho}{2}\|x - x_0\|^2= \langle s_k, x\rangle + c_k + \frac{\rho }{2}\|x - x_0\|^2 \quad \text{for some }s_k, c_k \in \mathbb{R}^n. \label{eq:Gamma-k-affine}
    \end{equation}
    The function above has the minimum $z_k$ by (c), where we necessarily have $z_k = z_0 - s_k/\rho$. 
    It follows that
    \begin{align}
        A_k\Gamma_k(x) + \frac{\rho}{2}\|x - z_0\|^2 & = \langle s_k, x\rangle + c_k + \frac{\rho}{2}\|x - z_k\|^2 - \langle s_k, x - z_k\rangle + \frac{1}{2\rho}\|s_k\|^2 \label{eq:Gamma-k-expand}\\
        & = \langle s_k, z_k\rangle + c_k + \frac{\rho}{2}\|x - z_k\|^2 + \frac{\rho}{2}\|z_k - x_0\|^2 \label{eq:Gamma-k-collect}\\
        & = A_k\Gamma_k(z_k) + c_k + \frac{\rho}{2}\|x - z_k\|^2 + \frac{\rho}{2}\|z_k - x_0\|^2, \label{eq:Gamma-k-final}
    \end{align}
    where \Cref{eq:Gamma-k-expand} uses the decomposition \Cref{eq:Gamma-k-affine} and $z_k = z_0 - s_k/\rho$, \Cref{eq:Gamma-k-collect} collects terms and applies the relation $z_k = z_0 - s_k/\rho$ again, and \Cref{eq:Gamma-k-final} recalls \Cref{eq:Gamma-k-affine}. This completes the proof of (d).
    \end{enumerate}
\end{proof}

Points (a) and (b) in \cref{lemma:affine-maps} show that the functions $\gamma_{k+1}$ and $\Gamma_k$ are affine and lower than the function $f$. Point (c) reveals that the iterate $z_k$ is in fact the minimizer of the function $A_k \Gamma_k(\cdot) + \frac{\rho}{2}\|\cdot - z_0\|^2.$ Point (d) is an identity that will be useful later.

The next lemma establishes another useful identity.
\begin{lem} \label{lemma:seq-relation}
    The sequences $\{a_k\}$ and $\{A_k\}$ in \cref{alg:A-HPE-simplified} satisfy $A_{k + 1} = a_{k }^2$ for all $k \geq 0$.
\end{lem}
\begin{proof}
    From the update of $a_k$, it is clear that $a_k$ is a positive root of the equation 
    \begin{align*}
        a^2  - a -A_k  = 0,
    \end{align*}
    where $A_k$ is given. Plugging $a_k$ into the above equation and rearranging terms gives us
     $   a^2_k = A_k + a_k = A_{k+1},$ where the last equality comes from the update of $A_{k+1}$.
\end{proof}

We next review an identity without proof, as it is independent of the algorithm. 
\begin{lem}[{\citet[Lemma 3.3]{monteiro2013accelerated}}] \label{lemma:inexact-lem}Take $\tilde{x},\tilde{y},\tilde{g} \in \mathbb{R}^n$ and $\rho,\epsilon > 0$. Then the inequality
\begin{equation}
    \|\frac{1}{\rho}\tilde{g} + \tilde{y} - \tilde{x}\|^2 + \frac{2}{\rho}\epsilon \leq \|\tilde{x} - \tilde{y}\|^2 \label{eq:inexact-crit}
\end{equation}
holds if and only if
\begin{equation*}
    \min_{x \in \mathbb{R}^n}\left\{\langle\tilde{g}, x - \tilde{y}\rangle - \epsilon + \frac{\rho}{2}\|x - \tilde{x}\|^2\right\} \geq 0. \label{eq:inexact-equiv}
\end{equation*}
\end{lem}

Recall the inexactness condition \cref{eq:A-HPE-condition} in \cref{alg:A-HPE-simplified}:
\begin{align*}
    x_{k+1} = y_k - \frac{1}{\rho} g_{k + 1}, \; g_{k + 1} \in \partial_{\epsilon_{k+1}}f(x_{k+1}),  \; 2 \epsilon_{k+1} \leq \rho \| x_{k+1} - y_k\|^2, \; \forall k \geq 0.
\end{align*}
Using \cref{lemma:inexact-lem} with $\tilde{g} = g_{k + 1}$, $\tilde{x} = y_k$ and $\tilde{y} = x_{k+1}$, we then know that 
\begin{align}
    \label{eq:nonnegative}
     \min_{x \in \mathbb{R}^n}\left\{\langle g_{k + 1}, x -x_{k+1}\rangle - \epsilon_{k+1} + \frac{\rho}{2}\|x - x_{k+1}\|^2\right\} \geq 0.
\end{align}
The following lemma establishes the nondecreasing property of the sequence $\{\beta_k\}$.

\begin{lem}
The sequence $\{\beta_k\}_{k = 1}^{\infty}$ defined in \cref{eq:beta-k-def} satisfies $0\leq \beta_k \leq \beta_{k + 1}, \forall k \geq 0$.
\end{lem}
\begin{proof}
    As $A_0 = 0$, one easily has that
    \begin{equation*}
        \beta_0 = \inf_{x \in \mathbb{R}^n} \left\{\frac{\rho}{2}\|x - x_0\|^2 \right\} = 0.
    \end{equation*}
    We move on to proving $\beta_{k + 1} \geq \beta_k$ when $k > 0$. For any $y \in \mathbb{R}^n$, define
    \begin{equation*}
        \tilde{y} := \frac{A_k}{A_{k + 1}}x_k + \frac{a_{k }}{A_{k + 1}}y.
    \end{equation*}
    By the update of $y_k = \frac{A_k}{A_{k+1}} x_k + \frac{a_{k}}{A_{k+1}} z_k$, the update $A_{k + 1} = A_k + a_{k }$, and the fact that $\gamma_{k + 1}$ is affine, the following two facts hold:
    \begin{align}
        \tilde{y} - y_k & = \frac{a_{k }}{A_{k + 1}}(y - z_k), \label{eq:tilde-x-rel}\\
        \gamma_{k + 1}(\tilde{y}) & = \frac{A_k}{A_{k + 1}}\gamma_{k + 1}(x_k) + \frac{a_{k }}{A_{k + 1}}\gamma_{k + 1}(y). \label{eq:gamma-k+1}
    \end{align}
    By the definition \Cref{eq:Gamma-k-def}, we observe that for all $y$ we have
    \begin{align}
        A_{k + 1}\Gamma_{k + 1}(y) + \frac{\rho}{2}\|y - z_0\|^2 & = a_{k}\gamma_{k + 1}(y) + A_k\Gamma_k(y) + \frac{\rho}{2}\|y - z_0\|^2 \nonumber\\
        & = a_{k }\gamma_{k + 1}(y)  + A_k\Gamma_k(z_k) + \frac{\rho}{2}\|z_k - z_0\|^2 + \frac{\rho}{2}\|y - z_k\|^2 \label{eq:rewrite}\\
        & = a_{k}\gamma_{k + 1}(y) + A_kf(x_k) + \beta_k + \frac{\rho}{2}\|y - z_k\|^2, \label{eq:beta-k-sup}
    \end{align}
    where \cref{eq:rewrite} applies (d) in \Cref{lemma:affine-maps}, and \Cref{eq:beta-k-sup} is due to (c) of \Cref{lemma:affine-maps} and the definition of $\beta_k$ \cref{eq:beta-k-def}. Now, as $\gamma_{k + 1} \leq f$ by \Cref{lemma:affine-maps}, we may write
    \begin{align}
        A_{k + 1}\Gamma_{k + 1}(y) + \frac{\rho}{2}\|y - z_0\|^2 & \geq \beta_k + a_{k }\gamma_{k + 1}(y) + A_k\gamma_{k + 1}(x_k) + \frac{\rho}{2}\|y - z_k\|^2 \nonumber
        \\
        & = \beta_k + A_{k + 1}\gamma_{k + 1}(\tilde{y}) + \frac{\rho A_{k+1}^2 }{2a_{k }^2}\|\tilde{y} -y_k\|^2 
        \nonumber
        \\
        & = \beta_k + A_{k + 1}\gamma_{k + 1}(\tilde{y})+ \frac{\rho A_{k + 1}}{2}\|\tilde{y} - y_k\|^2, \label{eq:lemma-res}
    \end{align}
    where the first equality applies \Cref{eq:tilde-x-rel,eq:gamma-k+1}, and the second equality uses $A_{k+1} = a_k^2$ from \Cref{lemma:seq-relation}. Evaluating $\gamma_{k + 1}$ at $\tilde{y}$ and applying \cref{eq:nonnegative} shows
    \begin{align}
         \gamma_{k + 1}(\tilde{y}) + \frac{\rho}{2}\|\tilde{y} - y_k\|^2
        & = f(x_{k + 1}) + \left(\langle g_{k + 1}, \tilde{y} -x_{k + 1} \rangle - \epsilon_{k + 1} + \frac{\rho}{2}\|\tilde{y} - y_k\|^2 \right) \nonumber\\
        & \geq f(x_{k + 1}). \label{eq:gamma-k-bound}
    \end{align}
    As $A_{k + 1}$ is nonnegative, substituting \Cref{eq:gamma-k-bound} into \Cref{eq:lemma-res} and taking the infimum over $y$ shows
    \begin{equation*}
    \begin{aligned}
        \beta_k + A_{k + 1}f(x_{k + 1} ) & \leq \inf_{y \in \mathbb{R}^n}\left(A_{k + 1}\Gamma_{k + 1}(y) + \frac{\rho}{2}\|y-z_0\|^2 \right) \\
    &  = \beta_{k+1} +A_{k + 1}f(x_{k + 1}) ,
    \end{aligned}
    \end{equation*}
    where the equality is from \cref{eq:beta-k-def}. Subtracting $A_{k+1}f(x_{k+1}) $ on both sides of the above inequality finishes the proof.
\end{proof}

We are ready to fully establish \cref{eq:key-2}. As $\beta_k \geq 0$, it follows that
\begin{equation*}
    \begin{aligned}
        A_kf(x_k)  & \leq \inf_{x' \in \mathbb{R}^n}\left(A_k\Gamma_k(x')  + \frac{\rho}{2}\|x' - z_0\|^2\right) \\
        & = A_k\Gamma_k(z_k)  + \frac{\rho}{2}\|z_k - z_0\|^2,
    \end{aligned}
\end{equation*}
where we applied (c) of \Cref{lemma:affine-maps} in the second line. Adding the quadratic $\frac{\rho}{2}\|x - z_k\|^2$ on both sides of the equation above yields the relation
    \begin{equation*}
    \begin{aligned}
        A_kf(x_k)  +\frac{\rho}{2}\|x - z_k\|^2 & \leq  A_k\Gamma_k(z_k)  + \frac{\rho}{2}\|z_k - z_0\|^2 +\frac{\rho}{2}\|x - z_k\|^2 \\
        & = A_k\Gamma_k(x)+ \frac{\rho}{2}\|x - z_0\|^2 \\
        & \leq A_kf(x)+ \frac{\rho}{2}\|x - z_0\|^2,
    \end{aligned}
    \end{equation*}
    where the equality in the second line applies (d) of \cref{lemma:affine-maps}, and the last line uses part (b) of \cref{lemma:affine-maps}.

\section{Alternative potential function-based proof for \cref{thm:monteiro-conv-main}}
\label{section:potential}
This section provides an alternative proof for the convergence of \cref{alg:A-HPE-simplified} (i.e., \cref{thm:monteiro-conv-main}). Our proof is motivated by \citet[Section 5.3]{d2021acceleration}. We use a potential function argument.
\begin{lem}
    \label{lemma:potential}
    Consider \cref{alg:A-HPE-simplified}. For all $x \in \mathbb{R}^n$ and $k \geq 1$, the following holds:
    \begin{align}
        \label{eq:potential}
        A_{k+1}(f(x_{k+1}) - f(x)) + \frac{\rho}{2}\|z_{k+1} - x\|^2 \leq A_{k}(f(x_{k}) - f(x)) + \frac{\rho}{2}\|z_{k} - x\|^2.
    \end{align}
\end{lem}
Before presenting the proof of \cref{lemma:potential}, we demonstrate how to use \cref{lemma:potential} to obtain \cref{thm:monteiro-conv-main}. Summing \cref{eq:potential} from $0$ to $k-1$ with $x = \argmin_{x \in S} \|x_0 - x\| $ and $S = \argmin_{x} f(x)$ yields
\begin{align*}
    A_{k}(f(x_{k}) - \min_{x} f(x))  \leq \frac{\rho}{2}\Dist(x_0,S)^2.
\end{align*}
Then, using \cref{eq:key-1} (i.e., $A_k \geq \frac{k^2}{4}, \forall k \geq 1$), we have
\begin{align*}
    f(x_{k}) - \min_{x} f(x) \leq \frac{2\rho \Dist(x_0,S)^2}{k^2}, \; \forall k \geq 1.
\end{align*}

Next, we prove an identity that is commonly used in the analysis of acceleration methods.
\begin{lem}[Key identity]
        \label{lemma:key-identity}
         Consider \cref{alg:A-HPE-simplified}. It holds that for all $x \in \RR^n$, we have
        \begin{align*}
        a_k\innerproduct{v_k}{x - x_{k+1}} + A_k \innerproduct{v_k}{x_k - x_{k+1}} 
        =  \frac{\rho}{2} \|z_{k+1} - x\|^2  + \frac{A_{k+1}}{2\rho} \|v_k\|^2 - \frac{\rho}{2} \| z_k - x \|^2. 
    \end{align*}
    \end{lem} 
    \begin{proof}
    Rearranging step 5 of \cref{alg:A-HPE-simplified}, we see
    $
        A_kx_k  = A_{k+1} y_k -  a_kz_k,
    $
    which leads to
    \begin{align}
        \label{eq:step-1}
        A_k \innerproduct{v_k}{x_k - x_{k+1}} =  \innerproduct{v_k}{A_{k+1} y_k -  a_kz_k - A_kx_{k+1}}.
    \end{align}
    It holds that
    \begin{equation}
        \label{eq:expension}
        \begin{aligned}
        &a_k\innerproduct{v_k}{x - x_{k+1}}  +   A_k \innerproduct{v_k}{x_k - x_{k+1}} \\
        \overset{(a)}{=} \quad & \innerproduct{v_k}{a_k x - a_k x_{k+1} + A_{k+1} y_k -  a_kz_k - A_kx_{k+1}} \\
        \overset{(b)}{=}\quad & \innerproduct{v_k}{a_k (x - z_k )+  A_{k+1}(y_k -  x_{k+1}) } \\
        \overset{(c)}{=} \quad & \innerproduct{v_k}{a_k (x- z_k )} + \frac{A_{k+1}}{\rho}\|v_k\|^2,
    \end{aligned}
    \end{equation}
    where $(a)$ uses \cref{eq:step-1}, $(b)$ applies $A_{k+1} = a_k + A_k$ and collects like terms, and $(c)$ uses $x_{k+1} = y_k - \frac{1}{\rho}v_k$.
    
    Let us look at the above inner product term
    \begin{equation}
        \label{eq:innerproduct}
        \begin{aligned}
            \innerproduct{v_k}{a_k (x - z_k) } &= \rho \innerproduct{\frac{a_k}{\rho}v_k}{ x - z_k } \\
        &\overset{(a)}{=} \frac{\rho}{2} \|\frac{a_k}{\rho}v_k +x - z_k \|^2 - \frac{\rho}{2} \| \frac{a_k}{\rho}v_k\|^2 - \frac{\rho}{2} \| x - z_k\|^2 \\
        & \overset{(b)}{=} \frac{\rho}{2} \|z_{k+1} - x\|^2  - \frac{a_k^2}{2\rho} \|v_k\|^2 - \frac{\rho}{2} \| z_k - x \|^2  \\
        & \overset{(c)}{=}\frac{\rho}{2} \|z_{k+1} - x\|^2  - \frac{A_{k+1}}{2\rho} \|v_k\|^2 - \frac{\rho}{2} \| z_k - x \|^2,
        \end{aligned}
    \end{equation}
    where $(a)$ uses the identity $2\innerproduct{v}{u} = \|v+u\|^2 - \|v\|^2 - \|u\|^2$ with $v = \frac{a_k}{\rho}v_k$ and $u = x- z_k $, $(b)$ uses the update $z_{k+1} = z_k - \frac{a_k}{\rho}v_k $, and $(c)$ uses the identity $A_{k+1} = a_k^2$ in \cref{lemma:seq-relation}.
    
    Putting \cref{eq:expension,eq:innerproduct} together finishes the proof.
    \end{proof}

We are ready to present the full proof of \cref{lemma:potential}.

\noindent \textbf{Proof of \cref{lemma:potential}}
    The proof follows from a weighted combination of two inequalities. Specifically, fix $x\in \RR^n$, since $v_k \in \partial_{\epsilon_{k+1}}f(x_{k+1})$, we have the following two ineuqalities 
    \begin{subequations}
        \begin{align}
        f(x_{k+1}) - f(x) + \innerproduct{v_k}{x- x_{k+1}} - \epsilon_{k+1} & \leq 0,     \label{eq:key-1-potential}       \\
        f(x_{k+1}) -f(x_k) + \innerproduct{v_k}{x_k - x_{k+1}} - \epsilon_{k+1}  &\leq 0.   \label{eq:key-2-potential}       
    \end{align}
    \end{subequations}
    To simplify the notation, we let $\delta_{k} = f(x_k) - f(x)$. Since $a_k, A_k \geq 0$, we have that 
    \begin{align*}
        a_k \times \cref{eq:key-1-potential}  + A_k \times\cref{eq:key-2-potential}  \leq 0.
    \end{align*}
    Writing out the terms in the weighted sum gives
    \begin{align*}
        a_k\delta_{k+1} &+ a_k\innerproduct{v_k}{x - x_{k+1}}   -a_k \epsilon_{k+1}  \\
        &+ A_k \delta_{k+1}  - A_k \delta_{k}  + A_k \innerproduct{v_k}{x_k - x_{k+1}} - A_k\epsilon_{k+1}  \leq 0.
    \end{align*}
    Grouping like terms and using the update $A_{k+1} = A_k + a_k$, we have 
    \begin{align}
    \label{eq:step-0}
        A_{k+1} \delta_{k+1} + a_k\innerproduct{v_k}{x - x_{k+1}} + A_k \innerproduct{v_k}{x_k - x_{k+1}} - A_{k+1} \epsilon_{k+1} \leq  A_k \delta_{k}.
    \end{align}
     Using \Cref{lemma:key-identity}, \cref{eq:step-0} becomes
    \begin{align*}
        A_{k+1} \delta_{k+1} +  \frac{\rho}{2} \|z_{k+1} - x\|^2  + \frac{A_{k+1}}{2\rho} \|v_k\|^2 -  A_{k+1} \epsilon_{k+1} \leq  A_k \delta_{k} + \frac{\rho}{2} \| z_k - x \|^2.
    \end{align*}
    It becomes clear that if we have $$\frac{A_{k+1}}{2\rho} \|v_k\|^2 -  A_{k+1} \epsilon_{k+1}  \geq 0, $$ then \cref{lemma:potential} is proved. Indeed, the condition \cref{eq:A-HPE-condition} is to ensure this, as we have 
    \begin{align*}
        A_{k+1} \epsilon_{k+1}  \leq \frac{\rho A_{k+1}}{2}\|x_{k+1} - y_k\|^2 = \frac{ A_{k+1}}{2\rho}\|v_k\|^2.
    \end{align*}
    Thus, the proof is complete. \hfill $\blacksquare$
\end{document}